\documentclass[a4paper,11pt]{amsart}

\usepackage[utf8]{inputenc}

\usepackage{amsmath, latexsym, graphicx, amssymb, amscd, mathrsfs}
\usepackage[all]{xy}

\usepackage{tikz}
\usetikzlibrary{shapes.geometric, intersections, calc}

\usepackage[british]{babel}
\usepackage{url,comment,array,caption}
\usepackage{appendix}

\usepackage{enumerate}

\usepackage{multirow}

\usepackage{color}

\definecolor{darkgreen}{rgb}{0,0.5,0} %
\usepackage[
        colorlinks, citecolor=darkgreen,
        backref,
        pdfauthor={Roberto Dvornicich, Francesco Veneziano, Umberto Zannier}, %
        pdftitle={Classification of rational angles in plane lattices}, %
]{hyperref}
\usepackage[alphabetic,backrefs]{amsrefs} %

\usepackage[draft]{fixme} %

\renewcommand{\epsilon}{\varepsilon}

\renewcommand{\Re}{\operatorname{Re}}
\renewcommand{\Im}{\operatorname{Im}}

\DeclareMathOperator{\Gal}{Gal}

\DeclareMathOperator{\PGL}{PGL}

\DeclareMathOperator{\GL}{GL}
\DeclareMathOperator{\Vol}{Vol}

\newcommand{\abs}[1]{\left| #1 \right|}
\newcommand{\generated}[1]{\left\langle #1 \right\rangle}
\newcommand{\generatedQ}[1]{\left\langle #1 \right\rangle_\mathbb{Q}}
\newcommand{\A}{\mathbb{A}}
\newcommand{\N}{\mathbb{N}}
\newcommand{\Q}{\mathbb{Q}}
\newcommand{\C}{\mathbb{C}}
\newcommand{\R}{\mathbb{R}}
\newcommand{\E}{\mathcal{E}}
\newcommand{\K}{\mathcal{K}}
\newcommand{\Z}{\mathbb{Z}}

\renewcommand{\P}{\mathbb{P}}
\newcommand{\Gemme}{\mathbb{G}_\mathrm{m}}
\newcommand{\radu}{U}

\newcommand{\Ci}{\mathcal{C}}

\newcommand*\circled[1]{\tikz[baseline=(char.base)]{
            \node[shape=circle,draw,inner sep=2pt] (char) {#1};}}

\numberwithin{equation}{section} %

\newtheorem{thm}{Theorem}[section]
\newtheorem{prop}[thm]{Proposition}
\newtheorem{lemma}[thm]{Lemma}

\theoremstyle{definition}

\newtheorem{example}[thm]{Example}

\theoremstyle{remark}
\newtheorem{rem}[thm]{Remark}

\title{Classification of rational angles in plane lattices}
\author{Roberto Dvornicich}
\address{Department of mathematics, University of Pisa, Largo Bruno Pon\-te\-cor\-vo~5, 56127 Pisa, Italy}
\email{roberto.dvornicich@unipi.it}

\author{Francesco Veneziano}
\address{Department of mathematics, University of Genova, Via Dodecaneso~35, 16146 Ge\-no\-va, Italy}
\email{veneziano@dima.unige.it}

\author{Umberto Zannier}
\address{Scuola Normale Superiore, Piazza dei Cavalieri 7, 56126 Pisa, Italy}
\email{umberto.zannier@sns.it}


\begin{document}

\begin{abstract}
This paper is concerned with configurations of points in a plane lattice  which determine angles that are rational multiples of $\pi$. We shall study how many such angles may appear in a given lattice and in which positions, allowing the lattice to vary arbitrarily.
 
This classification turns out to be much less simple than could be expected, leading even to parametrizations involving rational points on certain algebraic curves of positive genus. 
\end{abstract}

\maketitle

\section{Introduction}

The present paper will be concerned with lattices in $\R^2$ and in fact with the angles $\widehat{ABC}$  determined by an ordered triple of distinct  points $A,B,C$ varying through the lattice.  Our leading issue will involve angles which are rational multiples of $\pi$, which we will call  {\it rational angles} for brevity. These angles of course  appear in regular polygons, in tessellations of the plane and other similar issues, and it seems to us interesting  to study in which  lattices these angles appear and how. 

\medskip

Suppose $P,Q$ are points in the lattice $\Z^2$ and let $\widehat{POQ}$ be the angle formed between the rays $OP$ and $OQ$ (where $O$ denotes the origin); one might wonder when this angle is a rational multiple of $\pi$.  It turns out that, if $\theta$ is a rational multiple of $\pi$ then $\theta$ is one of $\pm\frac{\pi}{4},\pm\frac{\pi}{2},\pm\frac{3\pi}{4},\pm\pi$, as shown by J.~S.~Calcut in \cite{Calcut09} (see the Appendix to this Introduction for a proof simpler than Calcut's).  Analogous properties hold for the {\it Eisenstein lattice} generated by the vertices of an equilateral triangle in the plane: the rational angles which occur are precisely the integer multiples of $\pi/6$.  

\medskip

Let us see a few historical precedents of similar problems.

Considering  the simplest lattice $\Z^2$, E. Lucas   \cite{Lucas-Quinconces}  in 1878 answered in the negative the rather natural question  of  whether  points $A,B,C\in\Z^2$ can determine an equilateral triangle.  In fact, it is not very difficult to show that no angle $\widehat{ABC}$, with $A,B,C\in\Z^2$, can be equal to $\pi/3$.  (This will follow as an extremely special case of our analysis, but we anticipate it in the Appendix to the Introduction  with a short and simple argument.)

In 1946, in a related direction,  W. Scherrer
\cite{Scherrer-Gitter}   
considered   all regular polygons with all vertices  in a given arbitrary lattice, and proved that  the polygons which may occur are precisely  those with $3,4$ or $6$ sides. (See again the Appendix for an account of  Scherrer's nice proof with some comments.)

\medskip

These results prompt the question: {\it  what happens when we consider other plane\footnote{Of course one could consider analogous problems in higher dimensions, however we expect the complexity  will increase greatly. The results in \cite{kedlaya2020space} appear to be not unrelated.
} lattices?}

 Namely, for  an arbitrary  lattice $\Lambda\subset \R^2$, how can we describe all rational angles determined by three points in $\Lambda$?  The most ambitious goal would be to obtain  a complete classification, in some sense. In particular, we can list the following issues in this direction:
 
 \medskip 
 
{\bf 1.} Which rational angles can occur in a given plane lattice $\Lambda$?  Intuitively, we should not expect many such angles when the lattice is fixed. 

{\bf 2.}  In which positions can they occur? Namely, which triples of points $A,B,C\in\Lambda$ can determine such a given angle? Note that  it is clearly sufficient to consider the case when $B$ is  the origin, and we may replace $A,C$ by any two points in the lattice on the lines $OA,OC$. So we shall consider an angle determined by the origin and two lines passing through the origin and another lattice point.

{\bf 3.} Adopting a reciprocal  point of view,  how can we classify   plane lattices according to the `structure' of all rational  angles determined  by  their points?  
 
 \medskip

We shall give below a more precise meaning to these questions; for instance, we shall study plane lattices according to {\it how many} rational angles (with vertex in the origin)  may occur and in which geometric `configurations'; for instance it is relevant  which pairs of angles have a side in common.

It is clear that for any arbitrarily-prescribed angle we are able to find a plane lattice in which this angle appears as the angle determined by three points of the lattice. It is easily seen that even a second (rational) angle can be prescribed  arbitrarily. On adding further conditions on the rationality of other angles, and their relative positions, the arithmetical information deduced from these conditions increases and imposes severe restrictions on the lattice, which can lead to a classification. Let us observe that the variables in our problem are (i) the lattices, (ii) the (rational) angles, and (iii) the lines (through lattice points) determining these angles. We shall see that for three or more rational angles in a lattice we have roughly the following possibilities:
  
  \medskip
  
  Either (A) the angles belongs to a certain finite set which is described in Section~\ref{sec:minkowski}, or (B) the lattice belong to one of finitely many families of lattices (and corresponding angles) which are described in Section~\ref{sec:famiglie.descritte}.
  
  More precisely, with a notation that will be properly introduced later, we will prove in Sections~\ref{sec:minkowski} and \ref{sec:famiglie.parametriche} the following theorem:
\begin{thm}\label{thm:MAIN}
 Let $V$ be a space of one of the three types \circled{2}+\circled{2}+\circled{2}, \circled{3}+\circled{2}, or \circled{4}. Then either $V$ is homotetic to one of the spaces described in Section~\ref{sec:famiglie.descritte}, or any rational angle $\frac{a}{b}2\pi$ of $V$, with $a,b$ coprime, satisfies
 \[
  b\leq 2^7 \cdot 3^4 \cdot 5^3 \cdot\left(\prod _{7\leq p\leq 37}p\right)^2.
 \]
\end{thm}

 Furthermore, when the angles lie in the mentioned finite set and are fixed, then  
  
  Either A1: The lattice belongs to a certain well-described family (called CM in the paper); moreover,  for a fixed lattice in this family the points determining the angles are parameterized by the rational points in finitely many rational curves (see Section~\ref{section:equation:1angle} and equation~\eqref{eqn:1angle:seconda});
  
  or A2: Both the lattice and the vertices are parametrized by the rational points on finitely many suitable algebraic curves;
  
  or A3: We have a finite set of lattices-vertices.
  
  \medskip

We shall analyse all of these situations, also looking at the maximal number of different configurations of rational angles which can exist in a lattice, in the sense of 2. above. We shall prove that indeed this number is finite apart from well-described families. 

\medskip

As it is natural to expect, this study involves algebraic relations among roots of unity, a topic which falls into a well-established theory. However, even with these tools at disposal, our problem turns out to be  more delicate than can be expected at first sight,  in that some surprising phenomena  will appear on the way to a complete picture: for instance we shall see that some of the curves mentioned in the case A2 have positive genus, in fact we have examples of genus up to $5$. 

In this paper we shall give an `almost' complete classification, in the above alluded sense;  in particular, this shall be complete regarding the sets of angles which may appear. 
Concerning  the classification of   specific  configurations of  angles, we shall  confine them to a certain explicit finite  list. In the present paper we shall treat in full detail only a part of them, in particular when the rational angles considered share a side.

We also give a full discussion of a case in which configurations arise which correspond to rational points on a certain elliptic curve (of positive rank). (See \S~\ref{sec:esempio.genere.1})

We  postpone  to a second paper the discussion of the remaining few cases, which need not be treated differently, but are computationally rather complicated due to the combinatorics of the configurations. We note that in view of Faltings' theorem, there are only finitely many  rational points on the curves of genus $>1$ which appear; however no known method is available to calculate these points (only to estimate their number), so we shall not be able \textit{a priori} to be fully explicit in these cases.

\subsection*{A concrete example} As an illustration of the kind of problems one is faced with, we show here  an example of a curve of high genus which arises from the general treatment of lattices with three non-adjacent rational angles.

Fixing the amplitudes of the angles to be $\frac{3}{5}\pi,\frac{3}{10}\pi,-\frac{1}{10}\pi$ we are led to study the rational solutions to the system $f_1=f_2=0$, where
\begin{align*}
  f_1&=-a^2 b^2 - a^2 b c + a b^2 c - 2 a b c^2 + 2 a b^2 d - a^2 c d +{}\\
  &+10 a b c d - 5 b^2 c d - a c^2 d + b c^2 d - 4 a b d^2 + 2 b c d^2 - 
 c^2 d^2,\\
 f_2&=-a^2 b^2 - 2 a^2 b c + a b^2 c - 5 a b c^2 + 2 a^2 b d - a^2 c d +{}\\
 &+ 16 a b c d - 2 b^2 c d - 2 a c^2 d + b c^2 d - 8 a b d^2 + 
 2 a c d^2 - c^2 d^2.
\end{align*}
These equations define a variety in $\P_3$ which consists of the four lines $ab=cd=0$ and of an irreducible curve $\Ci$ of genus 5.

It is worth noting that the curve $\Ci$ contains some trivial rational points such as, for instance $(1:1:1:1)$, but also nontrivial rational points such as $(12:2:-8:-3)$, which corresponds to the lattice generated by 1 and $\tau=r\theta$, with
\begin{align*}
 r&=\frac{9}{2}+\frac{\sqrt{5}}{2}+\frac{1}{4}\sqrt{30+22\sqrt{5}}-\frac{1}{4}\sqrt{150+110\sqrt{5}}, & \theta&=e^{\frac{3}{5}\pi i}.
\end{align*}
In addition to the angle spanned by 1 and $\tau$, there are two more rational angle to be found between lines through the origin passing through elements of the lattice. The angle spanned by $\tau+12$ and $\tau+2$ has an amplitude of $\frac{3}{10}\pi$, and the one spanned by $\tau-3$ and $\tau-8$ has an amplitude of $\frac{1}{10}\pi$. Figure~\ref{fig:sporadico} illustrates these angles.

\begin{figure}
\begin{center}
\begin{tikzpicture}[scale=0.5]
\coordinate (O) at (0,0);
\node [anchor=90] at (O) {$0$};
\node [anchor=90] at (1,0) {$1$};
\node [anchor=270] at ({2.86807*cos(108)},{2.86807*sin(108)}) {$\tau$};
\node [anchor=270] at ({2.86807*cos(108)+2},{2.86807*sin(108)}) {$\tau+2$};
\node [anchor=270] at ({2.86807*cos(108)+12},{2.86807*sin(108)}) {$\tau+12$};
\node [anchor=270] at ({2.86807*cos(108)-8},{2.86807*sin(108)}) {$\tau-8$};
\node [anchor=270] at ({2.86807*cos(108)-3},{2.86807*sin(108)}) {$\tau-3$};

\foreach \y in {-10,-9,...,12} 
\node[circle, fill=black, inner sep=1pt, minimum size=1pt] at ({\y},{0}) {};

\foreach \y in {-9,-8,...,13} 
\node[circle, fill=black, inner sep=1pt, minimum size=1pt] at ({2.86807*cos(108)+\y},{2.86807*sin(108)}) {};

\foreach \y in {-8,-7,...,13} 
\node[circle, fill=black, inner sep=1pt, minimum size=1pt] at ({2*2.86807*cos(108)+\y},{2*2.86807*sin(108)}) {};

\foreach \y in {-7,-6,...,14} 
\node[circle, fill=black, inner sep=1pt, minimum size=1pt] at ({3*2.86807*cos(108)+\y},{3*2.86807*sin(108)}) {};

\draw [->,thick,red] (O)--(1,0);
\draw [->,thick,red] (O)--({2.86807*cos(108)},{2.86807*sin(108)});
\draw [->,thick,blue] (O)--({2.86807*cos(108)+2},{2.86807*sin(108)});
\draw [->,thick,blue] (O)--({2.86807*cos(108)+12},{2.86807*sin(108)});
\draw [->,thick,green] (O)--({2.86807*cos(108)-3},{2.86807*sin(108)});
\draw [->,thick,green] (O)--({2.86807*cos(108)-8},{2.86807*sin(108)});
\node[circle, fill=black, inner sep=1pt, minimum size=3pt] at ({0},{0}) {};
\end{tikzpicture}
\end{center}
\caption{}\label{fig:sporadico}
\end{figure}

 As will be clear soon, these problems are  better analysed not in plane lattices but viewing the real plane $\R^2$ as the complex field $\C$ and  considering, in place of the lattices,  the $\Q$-vector subspaces of $\C$ generated by the lattices.

\bigskip

\subsection*{Organization of the paper} The paper will be organized roughly as follows.

\medskip

- In \S~\ref{sec:definizioni}   we shall introduce in detail our issues, giving also some notation and terminology.    

\medskip

-   Section~\ref{sec:equazioni}   will be subdivided in several parts.   

In the first two parts we shall find general equations corresponding to the configurations that we want to study (as described in the previous  section).

The third part will be devoted to obtaining a linear relation in roots of unity, with rational coefficients, after elimination from the equations obtained formerly (depending on the configuration).

In the fourth part we will study more in depth the elimination carried out in the previous part, in order to prove some geometric results.

The fifth  part will recall known results from the  theory of linear relations in  roots of unity, which shall be used to treat the mentioned equations.

\medskip

- In \S~\ref{sec:simmetrie}  we shall study spaces with special symmetries (for instance those which correspond to imaginary quadratic fields).

\medskip

- In \S~\ref{sec:minkowski}  we shall prove the bound appearing in Theorem~\ref{thm:MAIN}; for this we shall use among other things results of the geometry of numbers (not applied to the original  lattice however, but to a certain region  in dimension three). 

\medskip

- In \S~\ref{sec:famiglie.parametriche} we classify the finite number of continuous families of lattices which escape the previous theorem. 

\medskip

- In \S~\ref{sec:4} we study the configurations  where there exist four non-proportional points  $P_1,P_2,P_3,P_4$  of the lattice   such that  every angle $\widehat{P_iOP_j}$ is rational.  In this analysis we shall meet two rather surprising geometric  shapes (which we shall call {\it dodecagonal}). 

\medskip

- Section \ref{sec:esempio.genere.1}  will contain the complete study of an elliptic curve such that its rational points correspond to lattices with three non adjacent rational angles. (The group of rational points will be found to be isomorphic to $\Z/(2)\times \Z$.)

\bigskip

\centerline{\bf Appendix to the Introduction}

\medskip

In this short Appendix we give a couple of simple proofs related to the known results cited in the Introduction. These will be largely  superseded by the rest of the paper, but due to their simplicity we have decided to offer independent short arguments for them.

\subsection{Rational angles in the Gaussian lattice.} By `Gaussian lattice' we mean as usual the lattice $\Lambda=\Z+\Z i\subset \C$. For our issue of angles, it is equivalent to consider the $\Q$-vector space generated by the lattice, i.e. $V=\Q +\Q i$. This space has the special feature of being  a field, i.e. the Gaussian field $\Q(i)$, which makes our problem quite simpler. In fact, let $\alpha=2\pi a/b$ be a rational angle occurring in $\Lambda$ or $V$, where $a,b$ are coprime nonzero integers, and let  $\zeta=\exp(2\pi ia/b)$. That $\alpha$ occurs as an angle in $\Lambda$  means that  there are  nonzero points $P,Q\in\Lambda\subset \C$ such that $Q=\zeta rP$ where $r\in\R^*$. Conjugating and dividing we obtain $\zeta^2=x/\bar x$, where $x= Q\bar P\in \Q(i)$. So the root of unity $\zeta^2$ lies in  $\Q(i)$ and it is well known that $\Q(i)$ contains only the fourth roots of unity, so $\zeta$ has order dividing $8$.  A direct argument is to observe that $\zeta^2$ is an algebraic integer so must be of the shape $r+is$ with integers $r,s$, and being a root of unity this forces $r^2+s^2=1$, so  $\zeta^2$ is a power of $i$ and  $\zeta^8=1$, as required. 

In the converse direction,  observe that indeed $\exp(\pi i/4)$ is determined by the three points $(1,1), (0,0), (1,0)$ of $\Lambda$,   (i.e. the complex numbers  $1+i, 0, 1 \in\Q(i)$).

A similar argument holds for the Eisenstein lattice, which again   generates over $\Q$ a field, namely the   field generated by the roots of unity of order $6$; hence $\zeta$ can be a $12$-th root of unity. 

These special lattices will be treated in greater generality in \S~\ref{section:CM}.

\medskip

\subsection{On Sherrer's proof.}  As mentioned above,   Sherrer proved that the only regular polygons with vertices in some lattice are the $n$-gons for $n=3,4,6$.  The  idea of his  proof is by descent: if $P_1,\ldots ,P_n$ are the vertices of an $n$-gon in a given lattice $\Lambda$, then $P_{j+1}-P_j$ are again the vertices of a regular $n$-gon in $\Lambda$; however,   if $n>6$, the new sides are  smaller whence iterating the procedure we obtain a contradiction  (and similarly with a little variation for $n=5$).  In fact, this argument amounts  to the following: we  assume  as before  that the lattice is inside $\C$ and (after an affine  map) that $\Q\Lambda$  contains all  the $n$-th roots of unity $\zeta^m$ where $\zeta=\exp(2\pi i/n)$.  By appealing to the fact, already known to Gauss, that the degree $[\Q(\zeta):\Q]$ is equal to $\phi(n)$, we may already conclude that $\phi(n)\le 2$, i.e. $n=1,2,3,4,6$. But we may also avoid using such result, on observing that for $n>6$, the ring $\Z[\zeta]$ contains nonzero elements of arbitrarily small complex absolute value, so $\Z[\zeta]$ cannot be contained in a lattice, which is discrete. To justify the assertion, consider the elements $(\zeta -1)^m$: if $n>6$ they are have absolute value which is decreasing to $0$ since $0<|\zeta-1|<1$. For $n=5$ we may instead consider $1+\zeta^2$ in place of $\zeta-1$.

\section*{Acknowledgements}
We thank Yves André, Julian Demeio and Davide Lombardo for useful comments and discussions.

We thank the anonymous referee for pointing out to us the references in Remark~\ref{rem:applicazione.euclidea}.

\section{Terminology and notation}\label{sec:definizioni}
We identify the euclidean plane with $\C$. We call a \textit{rational angle} in $\C$ an ordered couple of distinct lines through the origin such that the measure of the euclidean angle between them is a rational multiple of $\pi$. We say that two points $v_1,v_2\in \C\setminus\{0\}$ such that $v_2/v_1\not\in\R$ determine (or form) a rational angle if the lines $(\R v_1, \R v_2)$ do (i.e. if the argument $v_2/v_1$ is a rational multiple of $\pi$); with a slight abuse of notation we write the angle $(\R v_1, \R v_2)$ as $(v_1,v_2)$.

Let $\Lambda\subseteq\C$ be a lattice. Given a rational angle determined by elements of $\Lambda$, many more pairs in $\Lambda^2$ can be trivially found (by multiplication by integers) that determine the same angle. Therefore we prefer to tensor the whole lattice by $\Q$ and study angles in the tensored space. With this point of view, we can say that when we draw rational angles in $\Lambda$ we extend the sides indefinitely and we are not concerned with which points of $\Lambda$ actually meet the sides. 

In this setting, the objects that we will study are 2-dimensional $\Q$-vector spaces $V\subset\C$ that contain two $\R$-linearly independent vectors. These are precisely the sets obtained after tensoring a plane lattice by $\Q$. From now on, unless otherwise stated, we will refer to these sets simply as \textit{spaces}.

Any angle-preserving transformation of $\C$ that sends the origin to itself will clearly establish a bijection between the rational angles of a space $V$ and the rational angles of its image.

These transformations are generated by complex \textit{homotheties} of the form $z\mapsto \lambda z$ for a fixed $\lambda\in\C^*$ and by the complex conjugation.
For this reason we will say that two spaces $V_1,V_2$ are \textit{homothetic} (and we write $V_1\sim_h V_2$) if they are sent one to the other by a homothety, and we will say that they are \textit{equivalent} (and we write $V_1\sim V_2$) if $V_1$ is homothetic to $V_2$ or to its complex conjugate $\overline{V_2}$.

\begin{rem}
 Clearly every space is homothetic to a space containing 1. Given two spaces $V_1=\generatedQ{1,\tau_1}$ and $V_2=\generatedQ{1,\tau_2}$, it is easy to see that $V_1\sim_h V_2$ if and only if $\tau_2=\frac{a\tau_1+b}{c\tau_1+d}$, where $\left(\begin{smallmatrix}a & b\\ c & d\end{smallmatrix}\right)\in \PGL_2(\Q)$, in which case the homothetic coefficient $\lambda$ is given by $1/(c\tau_1+d)$.
 
 The same condition can also be expressed by saying that $V_1\sim_h V_2$ if and only if there is a $\Q$-linear dependence between $1,\tau_1,\tau_2, \tau_1\tau_2$; in fact, from $\tau_2=\frac{a\tau_1+b}{c\tau_1+d}$ a linear relation is immediately obtained, and vice versa, from such a linear relation we obtain a matrix which must be invertible because $\tau_2\not\in\Q$. 
 
 In conclusion, we can always assume up to homothety that $V=\generatedQ{1,\tau}$ with $\Im(\tau)>0$. Up to equivalence, we can additionally assume that $\Re(\tau)\geq 0$.
\end{rem}

For ease of notation, we will often write $V(\tau):=\generatedQ{1,\tau_1}$.

In every space $V$, given a rational angle $(v_1,v_2)$, we can obtain other rational angles by swapping them. We call the angles thus obtained $\textit{equivalent}$ to and we are not concerned with them.

Given two adjacent rational angles $(v_1,v_2)$ and $(v_2,v_3)$, we see immediately that $(v_1,v_3)$ is also a rational angle. For this reason it is more convenient to consider sets of adjacent angles as a single geometrical configuration, rather than as independent angles; this point of view is also supported by the shape of the equations that describe these cases. Therefore we call a \textit{rational $n$-tuple} a set of $n$ vectors $\{v_1,\dotsc,v_n\}$ such $(v_i,v_j)$ is a rational angle for all $i\neq j$. 

According to this definition, a rational $n$-tuple can be identified with an $n$-element subset of $\P(V)\subseteq \P(\C)$; this point of view however is not particularly useful when writing up the equations that describe the configuration.

\begin{rem}
 If a rational $n$-tuple and a rational $m$-tuple are not disjoint, then their union is still a rational $k$-tuple for some $k\geq n,m$.
\end{rem}

From the shape of equations \eqref{eq:tupla:tre} and \eqref{eq:nangle:due} below, which describe rational $n$-tuples, it is clear that angles which belong to a rational $n$-tuple containing 1 are qualitatively different from angles which do not belong to such an $n$-tuple. In fact, an angle in a rational $n$-tuple containing 1 leads to an equation of degree 1 in $\tau$, while angles in rational $n$-tuples not containing 1 lead to equations of degree 2 in $\tau$.

In light of these considerations, we will say that a space $V$ is of type \circled{$n$} if it contains a rational $n$-tuple; we extend this notation ``additively'', by saying that $V$ is of type \circled{$n$}+\circled{$m$} if 
it contains a rational $n$-tuple and a disjoint rational $m$-tuple, and so on. There is an obvious partial order on the possible types, and we say that $V$ has an exact type
$a_1$\circled{$n_1$}$+\dotsb+a_k$\circled{$n_k$}, if this type is maximal for $V$. We will characterize in Section~\ref{section:CM} the spaces for which such a maximal type exists.\\
We remark that spaces of type \circled{3} correspond to triangles in which all angles are rational multiples of $\pi$.

We will denote by $\radu$ the set of all roots of unity.

\section{Equations}\label{sec:equazioni}

\subsection{The equation of a rational angle}\label{section:equation:1angle}
Let $\tau\in\C\setminus \R$, and let $V=\generatedQ{1,\tau}$.
Let $a_0,a_1,b_0,b_1\in\Q$ be such that $(a_0\tau+a_1,b_0\tau+b_1)$ is a rational angle. That is to say, there exists a root of unity $\mu\in U\setminus\{\pm 1\}$ such that the ratio $\frac{b_0\tau+b_1}{a_0\tau+a_1}\mu$ is real. Setting this ratio equal to its conjugate leads to the equation
\begin{equation}\label{eqn:1angle:prima}
 (a_0\overline{\tau}+a_1)(b_0\tau+b_1)=\mu^2 (a_0\tau+a_1)(b_0\overline{\tau}+b_1).
\end{equation}
This shows that $\mu^2\in\Q(\tau,\overline{\tau})$.
Equation~\eqref{eqn:1angle:prima} is bi-homogeneous of degree 1 in both  $a_0,a_1$ and $b_0,b_1$, so it defines a curve $\Ci\subseteq\P_1\times\P_1$.
Setting
\begin{align*}
 A&=\tau\overline{\tau} & B&=\frac{\mu^2\tau-\overline{\tau}}{\mu^2-1} & C&=\frac{\mu^2\overline{\tau}-\tau}{\mu^2-1},
\end{align*}
we can rearrange equation \eqref{eqn:1angle:prima} as
\begin{equation}\label{eqn:1angle:seconda}
 a_0 b_0 A+a_0 b_1 B + a_1 b_0 C + a_1 b_1 =0.
\end{equation}
The curve $\Ci$ has genus 0 and it is irreducible. In fact, it has bidegree $(1,1)$ and, if it were not irreducible, it would have two components of bidegrees $(0,1)$ and $(1,0)$. 
This happens if and only if $A=BC$, and a small computation shows that this happens if and only if $\tau\in\R$.

There is a bijection between $\Ci(\Q)$ and the set of rational angles $(v_1,v_2)$ in $V$ with $\arg(v_2/v_1)=\mu$. However, $\Ci$ is in general not defined over $\Q$, but only over the field $\Q(\tau,\overline{\tau})\cap \R$. This fact plays a role in the characterization of spaces with infinitely many rational angles.

\subsection{The equations of a rational \texorpdfstring{$n$}{n}-tuple containing 1}\label{section:equations:tupla}
Let $V$ be a space with a rational $n$-tuple ($n\geq 3$). Up to homothety of the space and equivalence of rational angles we can assume that $V=\generatedQ{1,\tau}$, and that the $n$-tuple is given by $\{1,\tau,\tau+a_1,\dotsc,\tau+a_{n-2}\}$, where the $a_j$ are distinct rational numbers different from 0.
We write $\tau=r\theta_0$, with $r=\abs{\tau}$ and $\theta_0\in\radu$. Similarly, let $\tau+a_j=\abs{\tau+a_j}\theta_j$ for $j=1,\dotsc,n-2$.

In particular we have that $(r\theta_0+a_j)/\theta_j\in\R$ for $j=1,\dotsc,n-2$. Equating these numbers and their conjugates, we can write
\begin{align*}
 \frac{\tau+a_j}{\theta_j}&=\frac{\overline{\tau}+a_j}{\theta_j^{-1}}\\
 \frac{r\theta_0+a_j}{\theta_j}&=\frac{r\theta_0^{-1}+a_j}{\theta_j^{-1}}
\end{align*}
which we can solve for $\theta_j$, $r$ or $\tau$ obtaining 
\begin{align}
\label{eq:tupla:uno}x_j&=\frac{\tau+a_j}{\overline{\tau}+a_j} & & \\
\label{eq:tupla:due} r&=a_j \theta_0\frac{x_j-1}{x_0-x_j} & j&=1,\dotsc,n-2\\
\label{eq:tupla:tre}\tau&=a_j x_0\frac{x_j-1}{x_0-x_j}, & &
\end{align}
where we have set $x_j=\theta_j^2$ for $j=0,\dotsc,n-2$.

\subsection{The equations of a rational \texorpdfstring{$n$}{n}-tuple not containing 1}\label{section:equations:nangle}
Let $V=\generatedQ{1,\tau}$ and $\tau=r\theta_0$, with $r=\abs{\tau}$ and $\theta_0\in\radu$. Let $\{v_0,\dotsc,v_{n-1}\}$ be a rational $n$-tuple ($n\geq 2$) which does not contain vectors proportional to 1 or $\tau$. By rescaling we can assume $v_j=\tau+b_j$ for $j=0,\dotsc,n-1$ and the $b_j$ distinct non-zero rational numbers. Let $\mu_j\in\radu$ such that $\mu_j(\tau+b_0)/(\tau+b_j)\in\R$ for $j=1,\dotsc,n-1$. Then we have
\begin{align}
 \mu_j\frac{\tau+b_0}{\tau+b_j}&=\mu_j^{-1}\frac{\overline{\tau}+b_0}{\overline{\tau}+b_j}\\
\label{eq:nangle} y_j(\tau+b_0)(\overline{\tau}+b_j)&=(\overline{\tau}+b_0)(\tau+b_j)\\
 y_j(r\theta_0+b_0)(r\theta_0^{-1}+b_j)&=(r\theta_0^{-1}+b_0)(r\theta_0+b_j)\\
 y_j\left( r^2+r b_j \theta_0 +\frac{r b_0}{\theta_0}+b_0 b_j\right)&=r^2+\frac{r b_j}{\theta_0} +r b_0\theta_0+b_0 b_j,
 \end{align}
 where we have set $y_j=\mu_j^2$.
 Writing these equations as quadratic equations in $r$ we get
 \begin{align}
 \label{eq:nangle:uno}r^2(y_j-1)+\frac{r}{\theta_0}\big(b_0(y_j-x_0)+b_j(x_0 y_j -1)\big)+b_0 b_j (y_j-1)&=0\\
 \label{eq:nangle:due}\tau^2 (y_j-1)+\tau\left[b_0(y_j-x_0)+b_j(x_0 y_j -1)\right]+b_0 b_j x_0 (y_j-1)&=0
\end{align}
for $j=1,\dotsc,n-1$.

\subsection{Equations for the three main cases}\label{section:3casi}
We begin by observing that, if $V$ is a space of type \circled{2}, it is homothetic to $\generated{1,\theta r}$ with $\theta\in U\setminus\{\pm 1\}$.

If a space $V$ has two nonequivalent rational angles, we have seen in Sections~\ref{section:equations:tupla} and \ref{section:equations:nangle} that we can derive equations for $\tau$ with coefficients in cyclotomic fields, therefore $\tau\in\overline{\Q}$.

If we have two independent such equations, we can eliminate $\tau$ and obtain one equation in roots of unity with rational coefficients. We shall then apply the results of Section~\ref{section:unit-equation} with the aim of bounding the degree of the roots of unity intervening in the equation, outside of certain families which admit a parametrization.

When eliminating $\tau$ we could have two equations of shape \eqref{eq:tupla:tre} (which amounts to having a rational 4-tuple), or one of shape \eqref{eq:tupla:tre} and one of shape \eqref{eq:nangle:due} (which amounts to one rational triple and one more angle not adjacent to it) or two equations of shape \eqref{eq:nangle:due}, for which we need three rational angles pairwise non-adjacent.
These are the cases that we denote as type \circled{4}, type \circled{3}+\circled{2} and type \circled{2}+\circled{2}+\circled{2} and they will constitute the main equations, whose solutions we shall seek by means of Theorem~\ref{thm:unit.equation}. For ease of notation, we change variables here with respect to those considered in Sections \ref{section:equations:tupla} and \ref{section:equations:nangle}.

\medskip

Case \circled{4}

\medskip

In this case we have two equations of the shape \eqref{eq:tupla:tre}. Eliminating $\tau$ gives
\begin{equation}\label{equation:caso4:A}
(a_1-a_2)+a_2\frac{x_1}{x_0}-a_1\frac{x_2}{x_0}-a_1 x_1 +a_2 x_2+(a_1-a_2)\frac{x_1 x_2}{x_0}=0,
\end{equation}
where $a_0,a_1$ are distinct rational numbers different from 0 and $x_0,x_1,x_2$ are distinct roots of unity different from one. 

Now we set $x_0=x,x_1=y,x_2=z$ and $a_1=a,a_2=b$ and obtain

\begin{equation}\label{equation:caso4:B}
(a-b)x+by-az-a x y +b x z +(a-b) y z =0.
\end{equation}

\medskip

Case \circled{3}+\circled{2}

\medskip

In this case we have one equation of the shape \eqref{eq:tupla:tre} and one of the shape \eqref{eq:nangle:due}. Eliminating $\tau$ gives
\begin{multline*}%
 (2 {a_1}^2-{a_1} ({b_0}+{b_1}) +2 {b_0} {b_1})
 -{a_1} {b_0} {x_0}
 -{a_1}{b_1}\frac{1}{{x_0}}
 +{a_1}  ({b_0}-{a_1}){x_1} 
 +{a_1} ({b_1}-{a_1})\frac{1}{{x_1}}\\
 + {b_0}({a_1}-{b_1})\frac{ {x_0} }{{x_1}}+{b_1}({a_1}-{b_0})\frac{ {x_1} }{{x_0}}+ \left(-2 {a_1}^2+{a_1} ({b_0}+{b_1})-2 {b_0} {b_1}\right){y_1}\\
 +{a_1} {b_1} {x_0} {y_1}
  +\frac{{a_1} {b_0}{y_1}}{{x_0}}
 +{a_1} {x_1} {y_1} ({a_1}-{b_1})
 +\frac{{a_1} {y_1} ({a_1}-{b_0})}{{x_1}}\\
+{b_0} ({b_1}-{a_1})\frac{{x_1} {y_1} }{{x_0}}
+{b_1}({b_0}-{a_1})\frac{ {x_0} {y_1} }{{x_1}}=0,\\
\end{multline*}
where $a_1,b_0,b_1$ are rational, with $b_0\neq b_1$ and $a_1\neq 0$. Furthermore if one of $b_0,b_1$ is equal to $a_1$ or $0$ the configuration reduces to that of the case \circled{4}, so we may also assume that $a_1,b_0,b_1$ are all distinct and nonzero. The roots of unity $x_0,x_1,y_1$ are different from one, and $x_0\neq x_1$.

By setting $x_0=x,x_1=y,y_1=z$ and $a_1=a, b_0=b, b_1=c$ we write

\begin{multline}\label{equation:caso32:B}
\left(2 a^2-a (b+c)+2 b c\right) x y -a b x^2 y -a c y -a  (a-b)x y^2-a  (a-c)x\\
+ b  (a-c)x^2+c  (a-b)y^2-\left(2 a^2-a (b+c)+2 b c\right) x y z\\
+a c x^2 y z+a b y z+a (a-c)x y^2 z+a  (a-b)x z-c (a-b)x^2 z -b  (a-c)y^2 z=0.\\
\end{multline}

\medskip

Case \circled{2}+\circled{2}+\circled{2}

\medskip

In this case we have two equations of the shape \eqref{eq:nangle:due}. Eliminating $\tau$ gives an expression in 
the four rational parameters $b_0,b_1,c_0,c_1$ and the three roots of unity $x_0,y_0,z_0$, where $b_0\neq b_1$, $c_0\neq c_1$, $x_0,y_0,z_0\neq 1$. Furthermore, if the set $\{0,b_0,b_1,c_0,c_1\}$ contains fewer than five distinct elements, we are in a case treated previously, so we may also assume that the $b_0,b_1,c_0,c_1$ are all distinct and nonzero.

By setting $x_0=x,y_0=y,z_0=z$ and $b_0=a,b_1=b,c_0=c,c_1=d$ we obtain an unwieldy polynomial equation \begin{equation}\label{equation:caso222:B}P(x,y,z)=0\end{equation} of degree two in each of the three variables. For ease of reference, we list here in a table the coefficients of each term appearing in $P$.
\begin{equation}
\arraycolsep=1.5pt\def\arraystretch{2}
\label{table:caso222:coeff}
\begin{array}{|c|l|}
\hline
\text{Monomials}        &       \text{Coefficients}                                     \\
\hline
 x^2y^2z^2,1        &       -b (a - c) (b - d) d                                    \\
 x^2y^2z,z          &       b (a b c + a b d - 2 a c d - 2 b c d + c^2 d + c d^2)   \\
 x^2yz^2,y          &       d (a^2 b + a b^2 - 2 a b c - 2 a b d + a c d + b c d)   \\
 xy^2z^2,x          &       (a - c) (b - d) (a b + c d)                             \\
 x^2y^2,z^2         &       -b (b - c) c (a - d)                                    \\
 x^2yz,yz           &         -a^2 b c - a b^2 c - a^2 b d - a b^2 d + 8 a b c d - a c^2 d - 
 b c^2 d - a c d^2 - b c d^2\\
 x^2z^2,y^2         &       -a (b - c) (a - d) d                                    \\
 xy^2z,xz           &       \begin{aligned}-2 a^2 b^2 + a^2 b c + a b^2 c - 2 a b c^2 + a^2 b d + a b^2 d +\\+ 
 a c^2 d + b c^2 d - 2 a b d^2 + a c d^2 + b c d^2 - 2 c^2 d^2\end{aligned}\\
 xyz^2,xy           &       \begin{aligned}-2 a^2 b^2 + a^2 b c + a b^2 c + a^2 b d + a b^2 d - 2 a^2 c d +\\- 
 2 b^2 c d + a c^2 d + b c^2 d + a c d^2 + b c d^2 - 2 c^2 d^2\end{aligned}\\
 y^2z^2,x^2         &       -a (a - c) c (b - d)                                    \\
 x^2y,yz^2          &       c (a^2 b + a b^2 - 2 a b c - 2 a b d + a c d + b c d)   \\
 x^2z,y^2z          &       a (a b c + a b d - 2 a c d - 2 b c d + c^2 d + c d^2)   \\
 xy^2,xz^2          &       (b - c) (a - d) (a b + c d)                             \\
 xyz                &       \begin{aligned}2 (2 a^2 b^2 - a^2 b c - a b^2 c + 2 a b c^2 - a^2 b d - a b^2 d + 
   2 a^2 c d - 4 a b c d +\\+ 2 b^2 c d - a c^2 d - b c^2 d + 2 a b d^2 - a c d^2 - b c d^2 + 2 c^2 d^2)\end{aligned}\\
 \hline
\end{array}
\end{equation}
\begin{rem}
Notice that, while the shape of the equation \eqref{equation:caso222:B} is not perfectly symmetric in the three angles, by applying an homothety we can always easily permute the three angles. This amounts to saying that the equation is stable by the substitution
which sends the triples $(1,\tau),(\tau+a,\tau+b),(\tau+c,\tau+d)$ with angles (squared) $(x,y,z)$ to $(1,\tau'),(\tau'-1,\tau'-\frac{a}{b}),(\tau'+\frac{a-c}{c-b}),\tau'+\frac{a-d}{d-b})$ with angles (squared) $(y,x,z)$.
\end{rem}

We observe that the non-degeneracy conditions coming from the geometry of the problem ensure that the equations
\eqref{equation:caso4:B},\eqref{equation:caso32:B},\eqref{equation:caso222:B} do not identically vanish.

In cases \circled{4} and \circled{3}+\circled{2}, given a solution of equations \eqref{equation:caso4:B} and\eqref{equation:caso32:B} one can easily obtain the corresponding value of $\tau$ and geometric configuration from equation \eqref{eq:tupla:tre}: we have
\begin{equation*}
 \tau=ax\frac{y-1}{x-y}.
\end{equation*}
It is easy to check directly that, under the required conditions, this value of $\tau$ is never real.

In case \circled{2}+\circled{2}+\circled{2} we have two equations of degree two for $\tau$, and it is not always possible to go back from a solution of \eqref{equation:caso222:B} to a geometrical configuration.

Given a solution of \eqref{equation:caso222:B}, if the quantity $ab-cd$ is different from 0, then the two quadratic equations for $\tau$ are independent, and it is possible to solve for $\tau$ obtaining
\begin{equation*}
 \tau=\frac{(cd-ab)x(y-1)(z-1)}{a(y-x)(z-1)+b(xy-1)(z-1)-c(z-x)(y-1)-d(xz-1)(y-1)}
\end{equation*}

If instead the equality $ab=cd$ holds, either $\tau=0$ is the only common solution, or the two equations are proportional, and in this case two values of $\tau$ are found. We can study fully the cases in which this happens.

Setting $d=ab/c$ and eliminating it, we see that the two equations are proportional if the following unit equation is satisfied
 \begin{equation*}%
   -b (a - c) + c(a - c)  x + a (b - c) y - c(b - c)  x y - c(b - c)  z +  a (b - c) x z + c(a - c)  y z - b (a - c) x y z=0.
\end{equation*}
Writing $\theta^2=x, \mu^2=y,\eta^2=z$ and dividing by $-\theta\mu\eta$ gives
 \begin{equation}\label{equation:equazioni.proporzionali.2}
   b (a - c)\Re(\theta\mu\eta) - c(a - c) \Re\left(\frac{\mu\eta}{\theta}\right) - a (b - c) \Re\left(\frac{\theta\eta}{\mu}\right) + c(b - c)\Re\left(\frac{\theta\mu}{\eta}\right)=0.
\end{equation}
This is a rational combination of four cosines of rational multiples of $\pi$, and all such combinations have been classified in \cite{ConwayJones}, Theorem 7.

\subsection{Equations in roots of unity}\label{section:unit-equation}
The study of linear relations among roots of unity goes back to long ago.
For instance in 1877 Gordan \cite{Gordan} studied the equation
\[\cos x + \cos y + \cos z =-1\]
with $x,y,z$ rational angles, with the purpose of classifying
the finite subgroups of $\PGL_2$.
The matter was considered by several other authors, also studying
polygons with rational angles and rational side-lengths.
Among these authors we point out Mann \cite{Mann} and Conway and Jones \cite{ConwayJones}.
These last authors described these issues, important for the present 
paper, as ``trigonometric diophantine equations''.
We do not pause further on other references, but we remark that this
problem is linked to the conjectures of Lang on torsion points on 
subvarieties of tori.

For simplicity we state a theorem of \cite{ConwayJones}, which we will apply to the equations in roots of unity that we have obtained before.

\begin{thm}[Conway-Jones]\label{thm:unit.equation}
 Let
 \[
  \sum_{j=0}^{k-1}a_j \xi_j=0
 \]
be a linear relation with rational coefficients $a_j$ between roots of unity $\xi_j$, normalized with $\xi_0=1$.
Then either there is a vanishing subsum, or the common order $Q$ of the $\xi_j$ is a squarefree number satisfying
\[
 \sum_{p\mid Q}(p-2)\leq k-2.
 \]
\end{thm}
A generalization of this theorem, with a different proof, was given in \cite{DvoZan00}, and a version which takes into account reductions modulo prime numbers in \cite{DvoZan02}.

\section{Spaces with special symmetries}\label{sec:simmetrie}
We study now some special classes of spaces which are stable under some angle-preserving transformation. They are relevant to our program because the presence of such symmetries can lead to a richer set of rational angles.

\subsection{Spaces with \texorpdfstring{$V = \overline{V}$}{V=V conj}}\label{sec:autoconiugati}
The following lemma characterizes, up to homothety, the spaces which are fixed by the complex conjugation.
\begin{lemma}\label{lemma1}
 Let $V$ be a space. The following conditions are equivalent:
 \begin{enumerate}[(i)]
  \item\label{lemma1:i} $V$ is homothetic to $\generatedQ{1,\tau}$ with $\abs{\tau}=1$.
  \item\label{lemma1:ii} $V$ is homothetic to $\generatedQ{1,\tau}$ with $\tau\neq 0$ a purely imaginary number
  \item\label{lemma1:iii} $V$ is homothetic to a space $V'$ with $V'=\overline{V'}$.
 \end{enumerate}
\end{lemma}
\begin{proof}
 $\eqref{lemma1:i}\Leftrightarrow\eqref{lemma1:ii}$ We have that
 \[
  \generatedQ{1,\tau}=\generatedQ{\tau+1,\tau-1}\sim_h\generatedQ{1,\frac{\tau-1}{\tau+1}}
 \]
and $\tau\in S^1\setminus\{\pm 1\}$ if and only if $\frac{\tau-1}{\tau+1}$ is purely imaginary and non-zero.

\def\Immtheta{40}
\begin{center}
\begin{tikzpicture}[scale=1.5]
\coordinate (A) at (-1,0);
\coordinate (B) at (2.5,0);
\coordinate (C) at (0,-0.5);
\coordinate (D) at (0,1.5);
\coordinate (uno) at (1,0);
\node[anchor=90]at(uno){$1$};
\coordinate (tau) at ({cos(\Immtheta)},{sin(\Immtheta)});
\node[anchor=0]at(tau){$\tau$};
\coordinate (tau1) at ($(tau)+(1,0)$);
\node[anchor=180]at(tau1){$\tau+1$};
\coordinate (tau-1) at ($(tau)-(1,0)$);
\node[anchor=0]at(tau-1){$\tau-1$};
\coordinate (2tau) at ($2*(tau)$);
\node[anchor=180]at(2tau){$2\tau$};
\coordinate (O) at (0,0);
\node[anchor=45]at(O){$O$};
\draw [->,thin, name path=ab] (A)--(B);
\draw [->,thin, name path=cd] (C)--(D);
\draw [->,thick,red, name path=ac] (O)--(tau);
\draw [->,thick,red, name path=ac] (O)--(uno);
\draw [thin, name path=cd] (uno)--(tau1);
\draw [thin, name path=ac] (tau)--(tau1);
\draw [thin,dotted, name path=cd] (uno)--(tau);
\draw [->,thick,blue, name path=ac] (O)--(tau1);
\draw [->,thick,blue, name path=ac] (O)--(tau-1);
\draw [thin, name path=ac] (tau-1)--(2tau);
\draw [thin, name path=ac] (tau1)--(2tau);
\end{tikzpicture}
\end{center}

$\eqref{lemma1:ii}\Rightarrow\eqref{lemma1:iii}$ Clear.

$\eqref{lemma1:ii}\Leftarrow\eqref{lemma1:iii}$ Let $V'=\generatedQ{v_1,v_2}=\generatedQ{\overline{v_1},\overline{v_2}}$. This means that $\overline{v_1}=av_1+bv_2$ and $\overline{v_2}=cv_1+dv_2$ for some matrix $M=\left(\begin{smallmatrix}a & b\\ c & d\end{smallmatrix}\right)\in \GL_2(\Q)$, and $M^2=Id$.

The eigenvalues of $M$ can only be $\pm 1$. If they were equal, then $M$ would either be of infinite order, or be equal to  $\pm Id$, but this is impossible because otherwise $v_1,v_2$ would be both real or both purely imaginary.

Therefore $M$ has distinct eigenvalues $1$ and $-1$, and it can be diagonalized over $\Q$, that is to say that $V'=\generatedQ{w_1,w_2}$ with $w_1\in\R$ and $w_2\in i\R$. Now we have $V\sim_hV'\sim_h\generatedQ{1,w_2/w_1}$.
\end{proof}
It seems natural to compare the three properties (i), (ii), (iii) with 
the condition $V\sim_h\overline{V}$.
It is obvious that (iii) implies this condition.
However the converse implication does not hold: see Section~\ref{sec:omotetici_al_simmetrico}, 
especially Lemma~\ref{lemma2}, for this.

\subsection{CM spaces}\label{section:CM}
We say that a space $V$  has \textit{Complex Multiplication} if there is a $\lambda\in\C\setminus\Q$ such that the multiplication by $\lambda$ sends $V$ to itself. It is easy to see that this happens if and only if  $V\sim_h \Q(\tau)$ with $\tau$ imaginary quadratic, and in this case the homothetic coefficient $\lambda$ can be taken as any element in $V\setminus\Q$.

This shows that if the space is stabilized by a nontrivial homothety there are infinitely many other nontrivial homotheties which stabilize it, and the image of any rational angle under any such homothety is again a rational angle.

The following theorem summarizes the situation and fully describes the set of rational angles in a CM space, up to equivalence of angles and the action of the space on itself by multiplication.

\begin{thm}\label{thm:CM}
 Let $V$ be a space.
 \begin{enumerate}[(i)]
  \item\label{thm:CM:i} $V$ is CM if and only if $V\sim_h \Q(\sqrt{-d})$ for a squarefree $d\in\N$
  \item\label{thm:CM:ii} For a squarefree $d\neq 1,3$ the rational angles in $\Q(\sqrt{-d})$ are, up to equivalence, precisely those of the form $(v,\sqrt{-d}\cdot v)$ with $v\in\Q(\sqrt{-d})$
  \item\label{thm:CM:iii} The rational angles in $\Q(i)$ are, up to equivalence, precisely those of the form $(v,\lambda \cdot v)$ with $v\in\Q(i)$ and $\lambda=i,i+1,i-1$
  \item\label{thm:CM:iv} The rational angles in $\Q(\sqrt{-3})$ are, up to equivalence, precisely those of the form $(v,\lambda \cdot v)$ with $v\in\Q(\sqrt{-3})$ and $\lambda=\sqrt{-3},\zeta,\zeta-1,\zeta+1,\zeta+2$, where $\zeta=\frac{-1+\sqrt{-3}}{2}$.
 \end{enumerate}
\end{thm}
\begin{proof}
 In proving part \eqref{thm:CM:i} we can assume up to homothety (which preserves both conditions) that $V=\generatedQ{1,\tau}$ with $\tau\in\C\setminus\R$. As seen is section 1, 
 \[\tau=\frac{a\tau+b}{c\tau+d}\]
 for a $\left(\begin{smallmatrix}a & b\\ c & d\end{smallmatrix}\right)\in \PGL_2(\Q)$ if and only if the homothety with coefficient $1/(c\tau_1+d)$ sends $V$ to itself. If $\tau$ is quadratic irrational then the coefficients of its minimal polynomial provide the suitable $a,b,c,d$. Vice versa, if such a matrix exists we get immediately a quadratic polynomial satisfied by $\tau$.

Let us now prove part \eqref{thm:CM:ii}. Let $(v_0,v_1)$ be a rational angle. By considering the angle $(1,v_1/v_0)$ we can reduce to the case $v_0=1$. Then either $v_1$ is a rational multiple of $\sqrt{-d}$, which is what we need to prove, or $\{1,\sqrt{-d},v_1\}$ form a rational triple. Now we use the notation of Section~\ref{section:equations:tupla}. By equation \eqref{eq:tupla:uno} we see that $x_1$ is a root of unity in $\Q(\sqrt{-d})$, and therefore it must be $-1$. This means that $v_1$ lies on the imaginary axis, so it must indeed be a rational multiple of $\sqrt{-d}$.

 The proof of parts \eqref{thm:CM:iii} and \eqref{thm:CM:iv} is analogous. As before, $x_1$ must be a fourth (resp. sixth) root of unity, therefore $v_1$ lies on one of the lines whose angle with the real axis is an integral multiple of $\pi/4$ (resp $\pi/6$). The points $i,i+1,i-1$ (resp $\sqrt{-3},\zeta,\zeta-1,\zeta+1,\zeta+2$) lie on these lines, and any other such point must be a rational multiple of one of them. 
\end{proof}

We remark that if $V$ is obtained as $\Lambda\otimes\Q$ for a lattice $\Lambda=\generated{1,\tau}_\Z$, the condition that $V$ is a CM space in our sense is equivalent to saying that the elliptic curve $\C/\Lambda$ has Complex Multiplication.

\medskip

We have seen that CM spaces contain infinitely many rational angles. In fact they are the only spaces with this property.
\begin{thm}\label{thm:nonCM.angolifiniti}
Let $V$ be a non-CM space. Then $V$ has only finitely many rational angles.
\end{thm}
\begin{proof}
 Up to homothety, we can assume that $V=\generatedQ{1,\tau}$ for some $\tau\in\C\setminus\R$. We can also assume that $\tau$ is an algebraic number, otherwise, as remarked at the beginning of Section~\ref{section:3casi}, the space $V$ contains, up to equivalence, only one rational angle.  The number field $\Q(\tau,\overline{\tau})$ contains finitely many roots of unity; therefore it is enough to show that for every fixed root of unity $\mu\neq \pm 1$ such that $\mu^2\in \Q(\tau,\overline{\tau})$, the curve $\Ci\subseteq\P_1\times\P_1$ defined by equation~\eqref{eqn:1angle:seconda} has only finitely many rational points.
 
 With the notation of Section~\ref{section:equation:1angle}, if the three coefficients $A,B,C$ are not all rationals, then there exists a Galois automorphism $\sigma$ such that $\sigma(\Ci)\neq\Ci$. In this case, the set $\Ci(\Q)\subseteq \Ci\cap \sigma(\Ci)$ is contained in the intersection of two distinct irreducible curves of bidegree $(1,1)$ in $\P_1\times\P_1$ and therefore it  contains at most two elements.
 
 In the case that $A,B,C\in\Q$, the space $V$ is CM. In fact $A=\tau\overline{\tau}$ and $B+C=\tau+\overline{\tau}$, and this implies that $[\Q(\tau):\Q]=2$; by Theorem~\ref{thm:CM}~\eqref{thm:CM:i}, this is equivalent to $V$ being CM.
\end{proof}

We remark that the arguments of this proof allow one to bound the number of rational angles in $\generatedQ{1,\tau}$ in terms of $\abs{U\cap\Q(\tau,\overline{\tau})}$. We will see in Section~\ref{sec:minkowski} how this bound can be made independent of $\tau$.

We notice also that equation \eqref{equation:caso222:B}, in the case of a CM space not homothetic to $\Q(\sqrt{-1})$ or $\Q(\sqrt{-3})$ reduces to $-16(ab-cd)^2=0$, which defines a surface in $\P^3$.

\subsection{Spaces with \texorpdfstring{$V\sim_h \overline{V}$}{V hom V conj}}\label{sec:omotetici_al_simmetrico}
More in general, any reflection preserves angles so we can look at spaces which are stable by any reflection (not just the complex conjugation).
\begin{lemma}\label{lemma2}
 Let $V$ be a space. The following conditions are equivalent:
 \begin{enumerate}[(i)]
  \item\label{lemma2:i} $V\sim_h \overline{V}$
  \item\label{lemma2:ii} $V$ is homothetic to $\generatedQ{1,\tau}$ with $\abs{\tau}^2\in\Q$.
 \end{enumerate}
\end{lemma}
\begin{proof}
 $\eqref{lemma2:i}\Rightarrow\eqref{lemma2:ii}$ We have that $V=\lambda \overline{V}$ for some $\lambda\in\C^*$ and that $V=\lambda \overline{V}=\lambda \overline{\lambda \overline{V}}=\abs{\lambda}^2 V$, so that $\abs{\lambda}^2\in\Q$. Let now $v\in V$ such that $\lambda\overline{v}/v\not\in\R$; such a $v$ exists because $V$ contains two $\R$-linearly independent vectors. Then we have that $V=\generatedQ{v,\lambda\overline{v}}\sim_h\generatedQ{1,\lambda\overline{v}/v}$ and $\abs{\lambda\frac{\overline{v}}{v}}^2=\abs{\lambda}^2\in\Q$.
 
 $\eqref{lemma2:ii}\Rightarrow\eqref{lemma2:i}$ Both properties are invariant by homothety, so it is enough to show that $\generatedQ{1,\tau}\sim_h \overline{\generatedQ{1,\tau}}$ when $\abs{\tau}^2\in\Q$. Indeed \[\overline{\tau}\generatedQ{1,\tau}=\generatedQ{\overline{\tau},\abs{\tau}^2}=\generatedQ{\overline{\tau},1}=\overline{\generatedQ{1,\tau}}.\qedhere\]
\end{proof}
We will see later in Section~\ref{sec:dodecagonali} some relevant nontrivial examples of spaces with this property.

We remark that if $V$ is obtained as $\Lambda\otimes\Q$ for a lattice $\Lambda=\generated{1,\tau}_\Z$, the condition that $V\sim_h\overline{V}$ is equivalent to saying that the elliptic curve $\C/\Lambda$ is isogenous to its complex conjugate.
\begin{example}
Consider a space $V=\generatedQ{1, \tau}$,  with a  transcendental $\tau$ such that $\abs{\tau}^2=2$. We will show that $V$, which obviously satisfies the conditions of Lemma~\ref{lemma2}, does not satisfy the conditions of Lemma~\ref{lemma1}. 

Suppose that $V$ is homothetic to a space $\generatedQ{1,\omega}$ with $\abs{\omega}=1$ (as in 
condition (i) of Lemma~\ref{lemma1}).

Then it follows that
     \[\tau=\frac{a+b\omega}{c+d\omega}\] with $a,b,c,d$ rationals and $\omega$ transcendental.
The conditions $\abs{\omega}=1$ and $\abs{x}^2=2$ easily imply, writing $t=\Re\omega$, that
\[2(c^2+2cdt+d^2)=a^2+2abt+b^2,\]
which implies $a^2+b^2=2(c^2+d^2)$  and $ab=2cd$.

But this yields $(a\pm b)^2= 2(c\pm d)^2$, and since 2 is not a square we find that $a=b=c=d=0$, which is impossible.
\end{example}

\section{Finiteness of angles in lattices outside special families}\label{sec:minkowski}

Let us consider the equations obtained in Section~\ref{section:3casi}. They are of the form
\begin{equation}\label{eq:mink:generale}
 \sum_{e\in I} C_e x^{e_1} y^{e_2} z^{e_3}=0,
\end{equation}
where $e=(e_1,e_2,e_3)\in I=\{-1,0,1\}^3$ and the $C_e$ are homogeneous polynomials of degree 4 (in the most general case \circled{2}+\circled{2}+\circled{2}) in the four variables $b_0,b_1,c_0,c_1$, and we seek solutions $x,y,z\in \radu\setminus\{1\}$ and  $b_0,b_1,c_0,c_1$ distinct nonzero integers (the cases \circled{4} and \circled{3}+\circled{2} are of the same form, only involving fewer terms, fewer variables, and with coefficients $C_i$ of smaller degree).

Let us fix a solution $x,y,z,b_0,b_1,c_0,c_1$ and let $N$ be the minimal common order of the roots of unity $x,y,z$. We will argue now about the factorization of $N$, showing that either $N$ is bounded by an absolute constant, or the solution belongs to a parametric family of solutions corresponding to a translate of an algebraic subgroup of $\Gemme^3$ contained inside the variety defined by equation \eqref{eq:mink:generale}.

\subsection{Odd primes appearing with exponent 1}
Let $p$ be an odd prime dividing $N$ exactly.
Let $\zeta$ be a primitive $p$-th root of unity, and let us write $x=\zeta^{v_1}\zeta_1,y=\zeta^{v_2}\zeta_2,z=\zeta^{v_3}\zeta_3$, with $\zeta_1,\zeta_2,\zeta_3$ roots of unity of order coprime with $p$.
If we denote by $(v,e)$ the scalar product of $v=(v_1,v_2,v_3)$ and $e=(e_1,e_2,e_3)$ as vectors of $\R^3$, we can write the monomial $x^{e_1}y^{e_2}z^{e_3}$ as $\zeta^{(e,v)}\xi_e$ with $\xi_e$ a root of unity of order prime with $p$, and the equation can be rewritten as
\begin{equation}\label{eq:mink:sing:due}
 0=\sum_{i=0}^{p-1} \zeta^i \sum_{(v,e)\equiv i \pmod{p}} C_e \xi_e .
\end{equation}
It follows that the quantities 
\begin{equation*}%
 \gamma_i=\sum_{(v,e)\equiv i \pmod{p}} C_e \xi_e
\end{equation*}
for $i=0,\dotsc,p-1$ must all be equal because the fields generated by $\zeta$ and by the $\xi_e$'s are linearly disjoint.

Equation \eqref{eq:mink:generale} has only 27 terms. This implies that, for all primes bigger than 27, the coefficients $\gamma_i$ are all equal to zero. We consider only this case.

Not all entries of $v$ are multiples of $p$, otherwise $N$ would not be the exact order of $x,y,z$. Therefore the lattice $\Gamma= v\Z + p\Z^3\subseteq \Z^3$ is a lattice of volume $\Vol(\Gamma)=p^{2}$.

By classical results in the geometry of numbers, namely the exact value of Hermite's constant\footnote{This result goes back to Gauss through the theory of arithmetical reduction of ternary quadratic forms; see \cite{casselsIntroGeoNumbers} for a thorough treatment.

We thank Davide Lombardo for suggesting the use of the exact value of Hermite's constant in place of Minkowski's theorem: this leads to a considerable numerical improvement.} in dimension three, we obtain that there exists a $w\in\Gamma\setminus\{0\}$ of norm bounded as $\abs{w}\leq \left(\sqrt{2} p^{2}\right)^{1/3}$.

Let us write
\begin{equation*}
 0\neq w=kv+pv'.
\end{equation*}
Let $e',e''$ be two vectors intervening in the same $\gamma_i$. Then we have that
\[
 (w,e'-e'')\equiv (v,e'-e'')\equiv 0 \pmod{p}
\]
So either $(w,e'-e'')=0$ or $\abs{(w,e'-e'')}\geq p$, and then
\[
p\leq \abs{(w,e'-e'')}\leq \abs{w}\abs{e'-e''} 
\]
Notice that either $\abs{e'-e''}\leq 3$, or $e'-e''=(\pm 2,\pm 2,\pm2)$, in which case we can divide by 2 and replace $\abs{e'-e''}$ with $\frac{\abs{e'-e''}}{2}=\sqrt{3}$.
In both cases we obtain
\begin{align*}
 p &\leq 3 \abs{w} \leq 3\left(\sqrt{2} p^2\right)^{1/3},\\
 p &\leq 27\sqrt{2}<39.
\end{align*}

Then if $p\geq 39$ we have that $(w,e')=(w,e'')$ for every $e',e''$ intervening in the same $\gamma_i$.
Let us now consider a new triple  $(xt^{w_1},yt^{w_2},zt^{w_3})$, where $t$ is an unknown. If we substitute in \eqref{eq:mink:generale} and collect the powers of $\zeta$ as in \eqref{eq:mink:sing:due} we obtain
\begin{align*}
 0&=\sum_{i=0}^{p-1} \zeta^i \sum_{(v,e)\equiv i \pmod{p}} C_e \xi_e t^{(w,e)}=\sum_{i=0}^{p-1} \zeta^i t^{f_i} \gamma_i
\end{align*}
and we can collect the powers of $t$ for some exponents $f_i$. But the $\gamma_i$'s are all zero, and so $(xt^{w_1},yt^{w_2},zt^{w_3})$ is a solution identically in $t$.

\subsection{Primes appearing with exponent at least 2}
Let us fix a solution $x,y,z,b_0,b_1,c_0,c_1$, let $N$ be the minimal common order of the roots of unity $x,y,z$ and let $p^m$ be a prime power dividing $N$ exactly.
Let $\zeta$ be a primitive $p^m$-th root of unity, and let us write $x=\zeta^{v_1}\zeta_1,y=\zeta^{v_2}\zeta_2,z=\zeta^{v_3}\zeta_3$, with $\zeta_1,\zeta_2,\zeta_3$ roots of unity of order coprime with $p$.
As before, we can write the monomial $x^{e_1}y^{e_2}z^{e_3}$ as $\zeta^{(e,v)}\xi_e$ with $\xi_e$ a root of unity of order prime with $p$, and the equation can be rewritten as
\begin{equation}\label{eq:mink:pow:due}
 0=\sum_{i=0}^{p^{m-1}-1} \zeta^i \sum_{(v,e)\equiv i \pmod{p^{m-1}}} C_e \xi_e \zeta^{(v,e)-i}.
\end{equation}
The exponents $(v,e)-i$ in the sums on the right are all multiples of $p^{m-1}$, so those powers of $\zeta$ are $p$-th roots of unity, but the  degree of $\zeta$ over the field generated by the $\xi_e$ and the $p$-th roots of unity is exactly $p^{m-1}$, which implies that
\begin{equation}\label{eq:mink:pow:uno}
 \gamma_i=\sum_{(v,e)\equiv i \pmod{p^{m-1}}} C_e \xi_e \zeta^{(v,e)-i}=0
\end{equation}
for all $i=0,\dotsc, p^{m-1}-1$.
From now on, we can argue as we did previously in the case of $\gamma_i=0$.

Not all entries of $v$ are multiples of $p$, otherwise $N$ would not be the exact order of $x,y,z$, therefore the lattice $\Gamma= v\Z + p^{m-1}\Z^3\subseteq \Z^3$ is a lattice of volume $\Vol(\Gamma)=p^{2m-2}$.
As before, there exists a $w\in\Gamma\setminus\{0\}$ of norm bounded as $\abs{w}\leq \left(\sqrt{2} p^{2m-2}\right)^{1/3}$.

Let us write
\begin{equation*}
 0\neq w=kv+p^{m-1}v'.
\end{equation*}

If in any of the sums \eqref{eq:mink:pow:uno} two different $e',e''$ appear, we have that
\[
 (w,e'-e'')\equiv (v,e'-e'')\equiv 0 \pmod{p^{m-1}}
\]
So either $(w,e'-e'')=0$ or 
\[
p^{m-1}\leq \abs{(w,e'-e'')}\leq \abs{w}\abs{e'-e''}. 
\]
If $p$ is odd we can argue again that the only case in which $e'-e''$ has norm greater than 3 is when it is equal to $(\pm 2,\pm 2,\pm 2)$, in which case it can be replaced by its half, and we obtain
\begin{align*}
p^{m-1}&\leq \abs{w}3 \leq \left(\sqrt{2} p^{2m-2}\right)^{1/3} 3\\
p^{m-1}&\leq 27\sqrt{2} <39.
\end{align*}
If $p=2$ we can only use that $\abs{e'-e''}\leq 2\sqrt{3}$, which gives 
\begin{align*}
2^{m-1}&\leq \abs{w}2\sqrt{3} \leq \left(\sqrt{2} p^{2m-2}\right)^{1/3}2\sqrt{3}\\
2^{m-1}&\leq 24\sqrt{6}<59\\
m &\leq 6.
\end{align*}

Therefore if $p^{m-1}\geq 39$ (or $m>6$ for $p=2$) we have that $(w,e')=(w,e'')$ for every $e',e''$ intervening in the sums \eqref{eq:mink:pow:uno}.
Let us now consider a new triple  $(xt^{w_1},yt^{w_2},zt^{w_3})$, where $t$ is an unknown. If we substitute in \eqref{eq:mink:pow:due} we obtain
\begin{align*}
 0&=\sum_{i=0}^{p^{m-1}-1} \zeta^i \sum_{(v,e)\equiv i \pmod{p^{m-1}}} C_e \xi_e t^{(w,e)} \zeta^{(v,e)-i}=\sum_{i=0}^{p^{m-1}-1} \zeta^i t^{f_i} \gamma_i
\end{align*}
and we can collect the powers of $t$ for some exponents $f_i$. But we argued before that the $\gamma_i$'s are all zero, and so $(xt^{w_1},yt^{w_2},zt^{w_3})$ is a solution identically in $t$.

We have thus shown that every solution of \eqref{eq:mink:generale} either belongs to a parametric family given by $(xt^{w_1},yt^{w_2},zt^{w_3})$, which represents the units in a translate of an algebraic subgroup of $\Gemme$, or the common order $N$ is a divisor of 
\[
 N_0=2^6 \cdot 3^4 \cdot 5^3 \cdot\left(\prod _{7\leq p\leq 37}p\right)^2.
\]
Combined with the proof of Theorem~\ref{thm:nonCM.angolifiniti} this implies that, outside the parametric families just mentioned, which will be described in the next section, the number of rational angles in non-CM spaces is at most $2N_0$.

\section{Parametric families of solutions}\label{sec:famiglie.parametriche}
In this section we study which translates of algebraic subgroups of $\Gemme^3$ are contained in the variety defined by \eqref{eq:mink:generale}. These are the parametric families of solutions which escape the analysis carried out in the previous section.

Suppose then that we have such a solution. This amounts to setting $(x(t) ,y(t),z(t))=(x' \cdot t^m,y'\cdot t^p, z' \cdot t^q)$, where $x',y',z'\in U$, $t$ is a parameter, $x,y,z$ are not all constant in $t$, so we may assume $m,p,q$ not all zero and  coprime.
\subsection{Case 2+2+2}
Let us assume for now that we are in the most general case \circled{2}+\circled{2}+\circled{2}, that is to say that $a,b,c,d\in\Q$ are all distinct and non-zero and $x,y,z\neq 1$; in this case we may assume $m,p,q\geq 0$ by replacing $t$ with $t^{-1}$ and possibly exchanging the role of $a,b$ or $c,d$. 
\subsubsection{If \texorpdfstring{$m,p,q$}{m,p,q} are all positive}
In this case the term in $xyz$ is the one of highest degree (as polynomial in $t$) among those appearing in \ref{table:caso222:coeff}, so if \eqref{eq:mink:generale} is satisfied identically in $t$ its coefficient must be equal to 0. However this coefficient is $-b (a - c) (b - d) d$, which is not zero.

\subsubsection{If \texorpdfstring{$m=0$}{m=0} and \texorpdfstring{$p,q$}{p,q} are positive and distinct}
In this case the terms $y^2z^2,xy^2z^2 ,x^2y^2z^{2}$ are those of highest degree. The leading term in $t$ is equal to $(a - c) (b - d) (-c + b x') (a - d x')$, which tells us that, for the equation to be satisfied identically in $t$, we must have $x(t)=x'=-1$ and either $b=-c$ or $a=-d$. If this holds, the term of highest degree is either the one in $y^2z$ or the one in $yz^2$, but setting either coefficient equal to zero implies that both $b=-c$ and $a=-d$ hold. Under this further assumption, the whole polynomial reduces to $4ab(a-b)^2(y-z)^2$, which can't be identically 0 in $t$ if $p$ and $q$ are distinct.

\subsubsection{If $m=0$ and $p=q=1$}
Arguing as before we see that, for the term of degree 4 to be zero, we must have $x=-1$ and one of  $b=-c$ or $a=-d$. The part of degree 3 is now given by the two terms in $y$ and $z$ together; the coefficient to be set equal to 0 is given by
\[
 2a(b+c)\left( (a+b)(a+c)y'-(a-b)(a-c)z'\right).
\]
If $b=-c$, then the term in degree three vanishes, and looking at the terms of lower degree we reach what is indeed a parametric solution of \eqref{eq:mink:generale}:
\begin{equation}\label{eqn:famiglia1}
 x=-1, y=z= t, a=-d,b=-c.
\end{equation}

If $b\neq -c$, then we have
\[
 (a+b)(a+c)y'-(a-b)(a-c)z'=0,
\]
so $y'/z'=\pm 1$.

Setting $y'=z'$ implies $b=-c$, which was discussed above, while $y'=-z'$ implies $a^2=-bc$, which in turn reduces the full equation to
\[
 -16ac(a-b)^2 y'^2 =0,
\]
which is not the case.

\subsubsection{If $m=1$ and $p=q=0$}
In this case, looking at the coefficient of the term in $t^2$ gives
\begin{equation*}
\begin{split}
  abcd(c-a + (b-c) y'  + (a-d) z' +(d - b) y' z' ) \times \\
  \left(\frac{1}{b}-\frac{1}{d} + 
   \left(\frac{1}{d}-\frac{1}{a}\right) y'  +  \left(\frac{1}{c}-\frac{1}{b}\right) z' + \left(\frac{1}{a}-\frac{1}{c}\right)  y' z' \right)=0,
\end{split}
\end{equation*}
where the second factor is equivalent to the first after exchanging every one of $a,b,c,d$ with the reciprocals of $b,a,d,c$ respectively.
The unit equation
\[ (c-a) + (b-c) y'  + (a-d) z' +(d - b) y' z' =0
\]
doesn't have any one-term subsum equal to zero. Two-term subsums equal to zero are easily excluded, except possibly if
\begin{equation*}
\begin{cases}
 (c-a)+(d-b)y'z'=0\\
 (b-c)y'+(a-d)z'=0
\end{cases}
\end{equation*}
hold.
The second equation implies that $y'=\pm z'$ and then in turn the first one gives $y'^2=\pm 1$ (the two signs being chosen independently). Of these 4 alternatives, three lead immediately to a contradiction, and we are left with the conditions $y'=z', y'^2=1$, which implies $a+b=c+d$ and $y=z=-1$. But under these conditions the full equation reduces to $16 (a - c)^2 (b - c)^2 x$, which cannot be equal to zero.

We are left then with solutions of a 4-term equation in units without subsums equal to zero. Therefore $y',z'$ must be sixth roots of unity, and by direct inspection one finds three solutions (up to Galois action and exchanging the role of $y',z'$). Each of these solutions, when plugged back into \eqref{eq:mink:generale}, gives a non-vanishing term in degree 1.

\bigskip

Since we are assuming to be in the most general case  \circled{2}+\circled{2}+\circled{2}, we can freely permute the angles $x(t),y(t),z(t)$ by changing the base of our lattice and obtaining a new equation of the shape \eqref{eq:mink:generale} which is also identically satisfied. Then it is easy to see that we can always reduce up to homothety to one of the four cases discussed above.

\subsection{Case 3+2}

The equation is preserved by the transformations
\begin{align*}
 b \leftrightarrow c,z\mapsto 1/z\\
 a\mapsto -a, b\mapsto b-a, c\mapsto c-a, x\leftrightarrow y\\
 a\mapsto 1/a,b\mapsto 1/b, c\mapsto 1/c, x\mapsto 1/x, y\mapsto y/x.
\end{align*}
This allows us to assume $m,p,q\geq 0$ and $m\geq p$.

\subsubsection{If $m,p,q>0$ and $m>p$}
In this case the term of highest degree is $x^2yz$, whose coefficient is $ac$, which is different from zero.

\subsubsection{If $m,p,q>0$ and $m=p$}
In this case the terms of highest degree are $x^2yz$ and $xy^2z$, and setting the coefficient equal to zero leads to
\[
 acx'+a(a-c)y'=0.
\]
This implies $x'=\pm y'$, but $a\neq 0$ so we must have $x=-y$ and $a=2c$.
With this substitution, equation \eqref{eq:mink:generale} becomes (after cancelling a factor $4cx$)
\[
 -cx+(b-c)x^2-(b-c)z+cxz=0.
\]
Let us assume first that $m\neq q$. In this case we see that the term of highest degree is either $x^2$ (if $m>q$) or $xz$ (if $m<q$). Their coefficients are $b-c$ and $c$ respectively, so in either case they are not zero.

If instead $m=q$ the coefficient to be set equal to zero is  $(b-c)x'+cz'$. This implies $b-c=\pm c$, which is impossible because in one case we get $b=0$ and in the other $b=2c=a$.

\subsubsection{If $m>p>0$ and $q=0$}
In this case the terms of highest degree are $x^2yz$ and $x^2y$, so we get $acz'-ab=0$, therefore $z=-1$ and $b=-c$.
Substituting back into \eqref{eq:mink:generale} gives
\[
 a^2x-b^2x^2+2(a^2-b^2)xy-b^2y^2+a^2xy^2=0.
\]
Now if $m\neq 2p$ the dominant term is either $x^2$ (if $m>2p$) or $xy^2$ (if $m<2p$), and their coefficients are $-b^2$ and $a^2$ respectively, which are nonvanishing.
If instead $m=2p$ the coefficient to be set equal to zero is  $-b^2x'+a^2y'^2=0$. This implies $a^2=\pm b^2$, which gives $a=\pm b$ (because $a,b\in\Q$), so $a=-b=c$, which is a contradiction.

\subsubsection{If $m=p=1$ and $q=0$}
The terms of highest degree in $t$ are $x^2yz,x^2y,y^2xz,y^2x$. Setting the coefficient equal to zero gives
\[
 cx'z'-bx'+(a-c)y'z'+(b-a)y'=0.
\]
This is a four-term unit equation. After normalising the equation dividing by $x'$, we need to study the vanishing sub-sums.
The only non-trivial case is given by $c=a-b$ and $y=xz$, which gives a vanishing subsum. Substituting back into \eqref{eq:mink:generale} gives a nonvanishing coefficient either in degree 2 or in degree 1.

A simple computer check finds nine nontrivial solutions among the sixth roots of unity, but none of them leads to solutions identically in $t$.

\subsubsection{If $m,q>0$ and $p=0$}
The terms of highest degree are $x^2yz$ and $x^2z$, so we get $a y'+(b-a)=0$, so $y=-1$ and $b=2a$. Substituting back into \eqref{eq:mink:generale} gives
\[
 -cx-(2a-c)z+(2a-c)x^2+cxz=0.
\]
Let us assume first that $m\neq q$. In this case we see that the term of highest degree is either $x^2$ (if $m>q$) or $xz$ (if $m<q$). Their coefficients are $2a-c$ and $c$ respectively, so in either case they are not zero, because $2a=b$.
If instead $m=q$ the coefficient to be set equal to zero is  $(2a-c)x'+cz'$. This implies $2a-c=\pm c$, which is impossible because in one case we get $a=0$ and in the other $a=c$.

\subsubsection{If $m=p=0$ and $q=1$}
In this case the coefficient of the term of degree one in $t$ factors as
\[
 \left(-a+cx+(c-a)y\right)\left((b-a)x-by+axy\right).
\]
It is enough to set the first factor equal to zero, as the second one is completely analogous (as seen by the transformation $x\mapsto 1/x,y\mapsto 1/y,b\leftrightarrow c$).
This leads to a three-term unit equation. It is readily seen that there are no vanishing subsums, and a simple computer check finds no nondegenerate solution.

\subsubsection{If $m=1$ and $p=q=0$}
In this case the terms of highest degree are $x^2,yx^2,zx^2,yzx^2$, so we get the unit equation
\[
 b(a-c)-aby-c(a-b)z+acyz=0.
\]
We need to study the vanishing sub-sums. The only non-trivial case is given by $c=ab/(b-a)$ and $y=z=-1$, which indeed gives a vanishing subsum. Substituting back into \eqref{eq:mink:generale} gives a nonvanishing coefficient either in degree 1.
A simple computer check finds nine nontrivial solutions among the sixth roots of unity, but none of them leads to solutions identically in $t$.

\subsection{Case 4}

If we are in the case \circled{4} we notice that the symmetries of the equation allow us to permute freely the angles $x,y,z$. Furthermore, up to homotheties, we can assume that the three exponents $m,p,q$ are non-negative, thanks to the substitution $x \mapsto x^{-1}, y \mapsto y x^{-1} , z \mapsto z x^{-1}, a \mapsto 1/a, b \mapsto 1/b$.

\subsubsection{$m=p=0,q=1$}
In this case, setting the coefficients of the terms of degree zero and one equal to zero we obtain the system of equations
\begin{equation*}
\begin{cases}
 (a-b)x'+by'-ax'y'=0\\
 -a+bx'+(a-b)y'=0
\end{cases}
\end{equation*}
If we add them and divide by $-a$ we get
\[(1-x')(1-y')=0,\]
which is a contradiction.

\subsubsection{If $0\leq m<p\leq q$}
In this case, the term with the highest degree is $yz$, whose coefficient is $a-b$, which does not vanish.

\subsubsection{If $0<m=p<q$}
In this case the terms of highest degree are $xz$ and $yz$, and setting the coefficient equal to 0 leads to 
\[
 bx'z'+(a-b)y'z'=0,
\]
which implies $b=a-b$ and $x=-y$.
Substituting back into \eqref{equation:caso4:B} gives
\[
 2b(x^2-z)=0
\]
and indeed setting
\begin{equation}\label{eqn:famiglia2}
 a=2b, x=t, y=-t, z=t^2
\end{equation}
gives a parametric solution of \eqref{eq:mink:generale}.

\subsubsection{$m=p=q=1$}
Setting the coefficients of the terms in $t$ and $t^2$ equal to zero leads to 
\begin{equation*}
\begin{cases}
 (a-b)x'+by'-az'=0\\
 -ax'y'+bx'z'+(a-b)y'z'=0;
\end{cases}
\end{equation*}
multiplying the first equation by $y'$ and adding them leads to
\[
 b(y'-x')(y'-z')=0,
\]
which would imply either $y=x$ or $y=z$; this is a contradiction.

\subsection{Infinite families}\label{sec:famiglie.descritte}
We have found the infinite families \eqref{eqn:famiglia1} and \eqref{eqn:famiglia2}. We can now understand them better.

The family \eqref{eqn:famiglia1} is of type \circled{2}+\circled{2}+\circled{2} whenever the parameter $t$ is a root of unity. Its feature is the presence of a right angle; in fact the spaces of this family all belong to the self-conjugated spaces studied in Section~\ref{sec:autoconiugati}. Up to homothety they form a family parametrized by a rational number $a\neq 0,\pm 1$ and a root of unity $y\neq 1$, with $\tau$ a purely imaginary root of
\[
 \tau^2+(a-1)\frac{y+1}{y-1}\tau-a=0.
\]

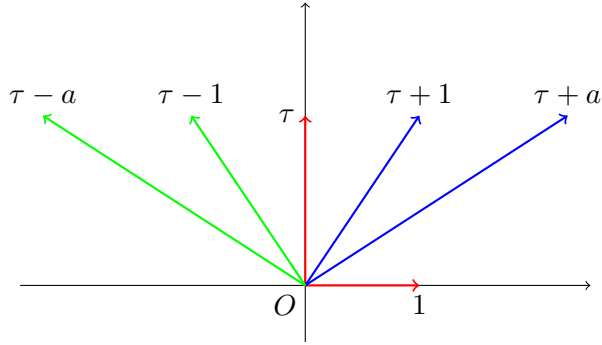
\def\Immtau{1.5}
\def\Imma{2.3}
\begin{center}
\begin{tikzpicture}[scale=1.5]
\coordinate (A) at (-2.5,0);
\coordinate (B) at (2.5,0);
\coordinate (C) at (0,-0.5);
\coordinate (D) at (0,2.5);
\coordinate (uno) at (1,0);
\node[anchor=90]at(uno){$1$};
\coordinate (tau) at (0,\Immtau);
\node[anchor=0]at(tau){$\tau$};
\coordinate (tau1) at ($(tau)+(1,0)$);
\node[anchor=270]at(tau1){$\tau+1$};
\coordinate (taua) at ($(tau)+(\Imma,0)$);
\node[anchor=270]at(taua){$\tau+a$};
\coordinate (tau-1) at ($(tau)-(1,0)$);
\node[anchor=270]at(tau-1){$\tau-1$};
\coordinate (tau-a) at ($(tau)-(\Imma,0)$);
\node[anchor=270]at(tau-a){$\tau-a$};
\coordinate (O) at (0,0);
\node[anchor=45]at(O){$O$};
\draw [->,thin, name path=ab] (A)--(B);
\draw [->,thin, name path=cd] (C)--(D);
\draw [->,thick,red, name path=ac] (O)--(tau);
\draw [->,thick,red, name path=ac] (O)--(uno);
\draw [->,thick,blue, name path=ac] (O)--(tau1);
\draw [->,thick,blue, name path=ac] (O)--(taua);
\draw [->,thick,green, name path=ac] (O)--(tau-1);
\draw [->,thick,green, name path=ac] (O)--(tau-a);
\end{tikzpicture}
\end{center}

If $a=0$ or $a=-1$, the type of the resulting space becomes \circled{4}, and we find the second parametric family \eqref{eqn:famiglia2}, which will be studied in full detail in the next section, as part of the complete description of all spaces of type \circled{4}.
Up to homothety, for a root of unity $y\neq 1$, $\tau$ is given by
\begin{align*}
\tau&=\frac{y+1}{y-1}.
\end{align*}

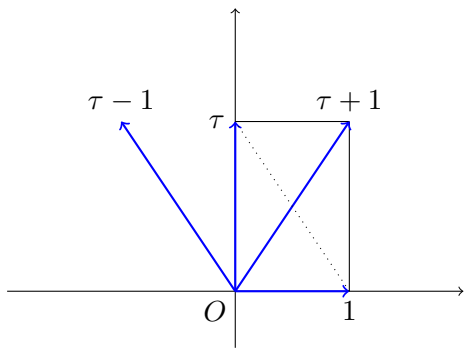
\def\Immtau{1.5}
\begin{center}
\begin{tikzpicture}[scale=1.5]
\coordinate (A) at (-2,0);
\coordinate (B) at (2,0);
\coordinate (C) at (0,-0.5);
\coordinate (D) at (0,2.5);
\coordinate (uno) at (1,0);
\node[anchor=90]at(uno){$1$};
\coordinate (tau) at (0,\Immtau);
\node[anchor=0]at(tau){$\tau$};
\coordinate (tau1) at ($(tau)+(1,0)$);
\node[anchor=270]at(tau1){$\tau+1$};
\coordinate (tau-1) at ($(tau)-(1,0)$);
\node[anchor=270]at(tau-1){$\tau-1$};
\coordinate (O) at (0,0);
\node[anchor=45]at(O){$O$};
\draw [->,thin, name path=ab] (A)--(B);
\draw [->,thin, name path=cd] (C)--(D);
\draw [->,thick,blue, name path=ac] (O)--(tau);
\draw [->,thick,blue, name path=ac] (O)--(uno);
\draw [thin, name path=cd] (uno)--(tau1);
\draw [thin, name path=ac] (tau)--(tau1);
\draw [->,thick,blue, name path=ac] (O)--(tau1);
\draw [->,thick,blue, name path=ac] (O)--(tau-1);
\draw [thin,dotted, name path=cd] (uno)--(tau);
\end{tikzpicture}
\end{center}

\begin{thm}\label{thm:finiti2}
 Every non-CM space has at most $2N_0+4$ different rational angles.
\end{thm}
\begin{proof}
Let us consider a space and assume that it has three nonequivalent  rational angles with arguments three roots of unity $x_1,x_2,x_3$ of order greater than $N_0$. Therefore, by the content of Section~\ref{sec:minkowski}, this space is homothetic to one of the families described above. However this is impossible, because it would imply that one of $x_1,x_2,x_3$ is equal to $-1$.
This shows that, in total, in a fixed space only $N_0+2$ different roots of unity can occur as arguments of a rational angle.
 It was already remarked in the proof of Theorem~\ref{thm:nonCM.angolifiniti} that, for a fixed root of unity $\mu$ and a fixed non-CM space, there are at most two nonequivalent rational angles of that argument, and this completes the proof.
\end{proof}

\section{Spaces with a rational 4-tuple}\label{sec:4}
\subsection{Rectangular and superrectangular spaces}
We saw before that the spaces which are invariant under complex conjugation can be characterized up to homothety as those generated by a $\tau$ of norm 1.
Among those spaces, the subfamily of those for which $\tau$ is a root of unity are especially relevant with respect to the rational angles.

\begin{lemma}\label{lemma:rect}
 Let $V$ be a space. The following conditions are equivalent:
 \begin{enumerate}[(i)]
  \item\label{lemma:rect:i} $V$ is homothetic to $\generatedQ{1,\tau}$ with $\tau\in \radu$
  \item\label{lemma:rect:ii} $V$ contains a rational triple with two perpendicular vectors
  \item\label{lemma:rect:iii} $V$ contains a rational $4$-tuple with two perpendicular vectors.
 \end{enumerate}
\end{lemma}
\begin{proof}
 The equivalence between \eqref{lemma:rect:i} and \eqref{lemma:rect:ii} is proved as in Lemma~\ref{lemma1}, and \eqref{lemma:rect:iii}$\Rightarrow$\eqref{lemma:rect:ii} is trivial, so only \eqref{lemma:rect:ii}$\Rightarrow$\eqref{lemma:rect:iii} needs to be shown:
 
  After applying an homothety and  multiplying the vectors in the triple by suitable rational numbers we can assume that the rational triple is $1,\tau,\tau+1$ for some purely imaginary number $\tau$. We see immediately that the stability of $V$ under complex conjugation allows us to add the vector $\tau-1=-(\overline{\tau+1})$ to the triple and obtain a rational quadruple.
\end{proof}

We define a space as \emph{rectangular} if it contains a rational angle of $\pi/2$, and \emph{superrectangular} if it satisfies the conditions of Lemma~\ref{lemma:rect}.

By definition every superrectangular space satisfies the hypotheses of Lem\-ma~\ref{lemma1}, whose points \eqref{lemma1:i} and \eqref{lemma1:ii} provide two different parametrizations of superrectangular spaces.

Looking for a purely imaginary $\tau$, with the notation in Section~\ref{section:equations:tupla}, we have $\theta_0=i,a_1=1,a_2=-1,\theta_2=-1/\theta_1$ and equation~\eqref{eq:tupla:tre} gives
\begin{align*}
\tau&=\frac{x_1-1}{x_1+1}
\end{align*}
and the 4-tuple is given by $\left(1,\frac{\theta_1^2-1}{\theta_1^2+1},\frac{2\theta_1^2}{\theta_1^2+1},\frac{-2}{\theta_1^2+1}\right)$.

Choosing a root of unity as a generator instead, we get $\tau=\theta_0,a_1=1,a_2=-1,x_1=\theta_0,x_2=-\theta_0$ and the 4-tuple is $\left(1,\theta_0,\theta_0+1,\theta_0-1\right)$.

\subsubsection{Rational $5$-tuples in superrectangular spaces}\label{section:5tuples.in.rectanglular}
Let us determine when the rational $4$-tuple of a superrectangular space can be extended to a rational $5$-tuple.
Let $\tau=\theta_0$ be a root of unity, let $V=\generatedQ{1,\theta_0}$ and $(1,\theta_0,\theta_0+1,\theta_0-1,\theta_0+a_3)$ be a rational $5$-tuple, with $x_1=\theta_0, x_2=-\theta_0$, always with the notation of Section~\ref{section:equations:tupla}. Then equation~\eqref{eq:tupla:tre} gives
\begin{align*}
 \theta_0&=a_3 x_0\frac{x_3-1}{x_0-x_3}\\
 \theta_0^2-x_3&=a_3 \theta_0 x_3-a_3 \theta_0.
\end{align*}
so we get the unit equation
\[ \theta_0-\frac{x_3}{ \theta_0}-a_3 x_3 + a_3=0.\]
No subsum can vanish, because $a_3\neq 0,\pm 1$ and $x_3\neq 1$, therefore by Theorem~\ref{thm:unit.equation} the only solutions are to be found among sixth roots of unity, and this case has been already discussed in Section~\ref{section:CM}. This proves that the only superrectangular spaces whose $4$-tuple can be extended to a $5$-tuple are those homothetic to $\Q(\sqrt{-3})$.

\subsubsection{Additional angles in superrectangular spaces}
Let us now check which superrectangular spaces have additional rational angles that are not part of a rational $n$-tuple containing 1.

Using the same notation for the space $V$, we seek a rational angle $(\tau+b_0,\tau+b_1)$. Then equation~\eqref{eq:nangle:due} gives us
\begin{equation*}
 (1+b_0 b_1)+b_0 \theta_0 + \frac{b_1}{\theta_0} -(1+b_0 b_1)y_1 - b_1 \theta_0 y_1-b_0\frac{y_1}{\theta_0}=0,
\end{equation*}
with $b_0,b_1$ distinct rationals, different from $0,1,-1$.

We are again in the position of using Theorem~\ref{thm:unit.equation}. To cut the number of cases to check, we see that the equation doesn't change if we swap $b_0$ and $b_1$ and invert $\theta$.
We also see that if $y_1=-1$ then $\theta$ has degree $2$ and we get one of the two CM superrectangular spaces $\Q(i)$ and $\Q(\sqrt{-3})$, which were already discussed.

A computer search doesn't find any solution with  common order a divisor of 30 but not of 6 (if the common order is a divisor of 6 we are again in the case of $\Q(\sqrt{-3})$).
We are left with examining all possible subsums.

The only coefficient that might vanish is $1+b_0 b_1$, but its vanishing implies $x_0=\pm 1$.

Of the fifteen two-term subsums, seven directly imply that $\theta_0$ or $y_1$ are equal to $\pm 1$; two imply that $\theta_0\in\Q(i)$; the remaining six lead to $\theta_0\in\Q(\sqrt{-3})$. 

Of the ten pairs of vanishing 3-term subsums, nine directly imply that $\theta_0\in\Q(\sqrt{-3})$ and the last one that $\theta_0\in\Q(i)$.

Therefore we can sum up these computations in the following statement:
\begin{thm}\label{thm:spazi_rettangolari}
 A superrectangular space not homothetic to  $\Q(i)$ or $\Q(\sqrt{-3})$ has exact type \circled{4}, i.e. it has only one rational 4-uple up to equivalence and no other rational angle.
\end{thm}

\subsection{The general case ``4''}
We have seen that every superrectangular spaces is of type \circled{4}. Let us now show that, with only finitely many exceptions, every space of type \circled{4} is a superrectangular space.

Let $V$ be a space with a rational 4-tuple given by $(1,\tau,\tau+a_1,\tau+a_2)$. By the computations in Section~\ref{section:equations:tupla} we have

\begin{align*}
\tau&=a_1x_0\frac{x_1-1}{x_0-x_1},\\
\tau&=a_2 x_0 \frac{x_2-1}{x_0-x_2}.
\end{align*}
eliminating $\tau$  we have
\begin{equation*}
  a_1(x_1-1)(x_0-x_2)=a_2(x_2-1)(x_0-x_1),
 \end{equation*}
 with $x_0,x_1,x_2\neq 1$ distinct roots of unity and $a_1,a_2\neq 0$ distinct rational numbers. In order to show that the space is superrectangular, it is enough to show that one of the $x_j$ or a ratio $x_j/x_k$ is equal to $-1$.
 
 We can rewrite this equation as
 \begin{equation}\label{eq:quadruple.razionali}
 (a_1-a_2)+a_2\frac{x_1}{x_0}-a_1\frac{x_2}{x_0}-a_1 x_1 +a_2 x_2+(a_1-a_2)\frac{x_1 x_2}{x_0}=0.
\end{equation}
We can now apply Theorem~\ref{thm:unit.equation} and conclude that either there is a vanishing subsum, or the common order of $x_0,x_1,x_2$ is a divisor of 30.

A computer search shows that there are no solutions with common order a divisor of 30 which do not belong to the family of superrectangular spaces. Therefore we are left with searching solutions with vanishing subsums.

\subsection{Vanishing subsums in \texorpdfstring{\eqref{eq:quadruple.razionali}}{I}}
If there is a vanishing subsum in \eqref{eq:quadruple.razionali}, then there is a vanishing subsum of minimal length at most three.

\subsubsection{One-term subsums}
If a vanishing subsum involves only one term, then the coefficient of the term must be 0, which is forbidden because $a_1,a_2$ are nonzero and distinct.

\subsubsection{Three-term subsums}
There are 20 three-term subsums, which get paired in 10 systems of two three-term linear equations. By direct examination, applying again Theorem \ref{thm:unit.equation}, one sees that any solution in roots of unity leads to a variable or ratio of two variables being equal to $\pm 1$, or to solutions where $x,y,z$ have common order a divisor of 6, which have been already discarded.

\subsubsection{Two-term subsums}
There are 15 two-term subsums. By direct inspection one checks that, after setting them equal to zero, 12 of them immediately imply that one variable or a ratio of two variables is equal to $\pm 1$. The remaining three, which are those obtained by pairing terms with the same coefficient, correspond the relations
\begin{align*}
 x_0& =-x_1 x_2 & x_1&=-x_0 x_2 & x_2&=-x_0 x_1.
\end{align*}
Each of these relations reduces \eqref{eq:quadruple.razionali} to a four-term equation. Precisely
\begin{align*}
 a_1 - a_1 x_1^2 + a_2 x_1 x_2 -a_2\frac{x_1}{x_2}&=0\\
 (a_1-a_2)-(a_1-a_2) x_1^2 + a_2 \frac{x_1}{x_0}-a_2 x_0 x_1 &=0\\
 (a_1-a_2) - (a_1-a_2) x_2^2 + a_1 x_0 x_2 -a_1\frac{x_2}{x_0}&=0.
\end{align*}
We can apply Theorem~\ref{thm:unit.equation} again, to find that solutions with vanishing subsums lead again to variables or ratios of variables being equal to $\pm 1$, while the solutions without vanishing subsums are found with $x_0^2,x_1^2,x_2^2$ of common order a divisor of 6, which implies that the common order of $x_0,x_1,x_2$ is a divisor of 12.

\subsection{Dodecagonal spaces}\label{sec:dodecagonali}
There are indeed solutions of common order 12.
Up to homotheties, exchanging the roles of the vectors in the rational $4$-tuple and acting with Galois automorphisms we find two spaces.

Let $\zeta=\exp(i \pi/6)$ be a primitive 12th root of unity.
The first space is given by
\begin{align}
\tau_1&=(i+1)\frac{\sqrt{3}-1}{2}=-\zeta^3+\zeta^2+\zeta-1  &\arg(\tau_1)&=\frac{1}{4}\pi\\
\tau_1+1&=(i+1)\frac{\sqrt{3}-i}{2}                                             &\arg(\tau_1+1)&=\frac{1}{12}\pi\\
\tau_1-1&=\frac{i-\sqrt{3}}{2}(\sqrt{3}-1)                        &\arg(\tau_1-1)&=\frac{5}{6}\pi,
\end{align}
as illustrated in Figure~\ref{fig:dodecagonali:tau1}.
The second is given by
\begin{align}
\tau_2&=(i-1)\frac{3-\sqrt{3}}{2}=\zeta^3-\zeta^2+\zeta-1  &\arg(\tau_2)&=\frac{3}{4}\pi\\
\tau_2+1&=\frac{1+i\sqrt{3}}{2}(\sqrt{3}-1)                         &\arg(\tau_2+1)&=\frac{1}{3}\pi\\
\tau_2+3&=\sqrt{3}(i-1)\frac{1+i\sqrt{3}}{2}                         &\arg(\tau_2+3)&=\frac{1}{12}\pi,
\end{align}
as illustrated in Figure~\ref{fig:dodecagonali:tau2}.
\begin{figure}
\begin{center}
\begin{tikzpicture}[scale=4]
\coordinate (O) at (0,0);

\node (pol) [draw, thick, blue!90!black,rotate=90,minimum size=8cm,regular polygon, regular polygon sides=12, rotate=195] at (0,0) {}; 
\coordinate (C) at (pol.corner 7);

\foreach \n [count=\nu from 0, remember=\n as \lastn, evaluate={\nu+\lastn}] in {1,2,...,12} 
\node[anchor=\n*(360/12)+180]at(pol.corner \n){$\zeta^{\nu}$};
\draw [thin, name path=d15] (pol.corner 1)--(pol.corner 5);
\draw [thin, name path=d27] (pol.corner 2)--(pol.corner 7);
\draw [thin, name path=d310] (pol.corner 3)--(pol.corner 10);
\draw [thin, name path=d412] (pol.corner 4)--(pol.corner 12);
\draw [thin, name path=d47] (pol.corner 4)--(pol.corner 7);
\draw [thin, name path=d612] (pol.corner 6)--(pol.corner 12);
\path [name intersections={of=d15 and d27,by=A}];
\path [name intersections={of=d47 and d612,by=B}];
\node [circle, fill=red, inner sep=0pt, minimum size=4pt,label=45:$\tau_1$] at (A) {};
\node [circle, fill=red, inner sep=0pt, minimum size=4pt, label=-90:$O$] at (O) {};
\draw [red, very thick] (O)--(A);
\draw [red, very thick] (A)--(B);
\draw [red, very thick] (B)--(C);
\draw [red, very thick] (C)--(O);

\end{tikzpicture}
\end{center}
\caption{}\label{fig:dodecagonali:tau1}
\end{figure}

\begin{figure}
\begin{center}
\begin{tikzpicture}[scale=4]
\coordinate (O) at (0,0);

\node (pol) [draw, thick, blue!90!black,rotate=90,minimum size=8cm,regular polygon, regular polygon sides=12, rotate=195] at (0,0) {}; 
\coordinate (C) at (pol.corner 7);
\foreach \n [count=\nu from 0, remember=\n as \lastn, evaluate={\nu+\lastn}] in {1,2,...,12} 
\node[anchor=\n*(360/12)+180]at(pol.corner \n){$\zeta^{\nu}$};
\draw [thin, name path=d14] (pol.corner 1)--(pol.corner 4);
\draw [thin, name path=d17] (pol.corner 1)--(pol.corner 7);
\draw [thin, name path=d27] (pol.corner 7)--($(pol.corner 2)!-1cm!(pol.corner 7)$);
\draw [thin, name path=d111] (pol.corner 11)--($(pol.corner 1)!-1cm!(pol.corner 11)$);
\path [name intersections={of=d27 and d111,by=B}];
\node [circle, fill=red, inner sep=0pt,minimum size=4pt, label=-75:$\tau_2+2$] at (B) {};
\draw [thin, name path=d39] (pol.corner 3)--(pol.corner 9);
\draw [thin, name path=d46] (pol.corner 4)--(pol.corner 6);
\draw [thin, name path=d57] (pol.corner 5)--(pol.corner 7);
\path [name intersections={of=d14 and d39,by=A}];
\path [name intersections={of=d46 and d57,by=tau2}];
\draw [thin] (tau2)--(B);
\draw [thin] (O)--(tau2);
\node [circle, fill=red, inner sep=0pt, minimum size=4pt] at (A) {};
\node [circle, fill=red, inner sep=0pt, minimum size=4pt, label=-90:$O$] at (O) {};
\node [circle, fill=red, inner sep=0pt, minimum size=4pt,label=-15:$\tau_2$] at (tau2) {};

\draw [red, very thick] (tau2)--(A);
\draw [red, very thick] (A)--(O);
\draw [red, very thick] (O)--(pol.corner 7);
\draw [red, very thick] (pol.corner 7)--(tau2);

\end{tikzpicture}
\end{center}
\caption{}\label{fig:dodecagonali:tau2}
\end{figure}
It is easy to check that $1,\tau_1,\tau_2$ and $\tau_1\tau_2$ are not $\Q$-linearly dependent, and thus that $V(\tau_1)$ is not homothetic to $V(\tau_2)$.

It is maybe more surprising that $\sigma(V(\tau_i))\sim_h V(\tau_i)$ for $i=1,2$ and every Galois automorphism $\sigma\in\Gal(\Q(\zeta)/\Q)$. 

In order to check this, let us fix $(1,\sqrt{3},i,i\sqrt{3})$ as a basis for $\Q(\zeta)/\Q$; in this basis, any Galois automorphism acts by exchanging the sign of some coordinates. Given two elements $v_1,v_2\in\Q(\zeta)$ and expressing in this basis the determinant of the $4\times 4$ matrix with column vectors $1,v_1,v_2,v_1 v_2$, one can check explicitly that it depends only on the squares of the coordinates of $v_1$ and $v_2$, and thus that it is invariant by the action of the Galois group. 

In particular this shows that the dodecagonal spaces satisfy the symmetry considered in Section~\ref{sec:omotetici_al_simmetrico}.

\subsubsection{Additional angles in dodecagonal spaces}\label{sec:altri_angoli_in_dodecagonali}
It is now quite easy to check that in these dodecagonal spaces the rational $4$-tuple that defines them cannot be extended to a rational $5$-tuple. 
With a little more (computational) work we also see that dodecagonal spaces do not contain any additional rational angle. Indeed it is enough to check which rational angles are contained in the spaces $\generatedQ{1,\tau_j}$ with $j=1,2$ and $\tau_j$ the ones defined in the previous section. 
By the usual \eqref{eq:nangle}, this amounts to finding all solutions of 
\[
 y(\tau_j+b_0)(\overline{\tau_j}+b_1)=(\overline{\tau_j}+b_0)(\tau_j+b_1)
\]
with $b_0,b_1\in\Q$ and distinct, and $y\in U$ different from 1. But we see immediately that $y$ lies in $\Q(\zeta_{12})$, the twelfth cyclotomic field, so $y=\zeta_{12}^{k}$ for $k=1,\dotsc,11$.
By writing both sides of the equation in terms of a basis, it is easy to see that the only solutions are found when $b_0,b_1\in\{0,1,-1\}$, for $j=1$, or $b_0,b_1\in\{0,1,3\}$, for $j=2$.

The computations of this section can be summarized in the following theorem:
\begin{thm}\label{thm:type4}
   Let $V$ be a space with a rational quadruple. Then $V$ is homothetic to either a superrectangular space or one of the 2 dodecagonal spaces.
\end{thm}
The following table summarizes the classification obtained so far.

\begin{center}
\begin{tabular}{|c|p{8cm}|c|}
\hline
Rational angles                  & Description & Type                  \\ \hline
$\infty\circled{2}$                  & $V$ homothetic to an imaginary quadratic field different from $\Q(\sqrt{-1})$ or $\Q(\sqrt{-3})$.                      & CM and Rectangular                  \\ \hline
$\infty\circled{4}$                  & $V$ homothetic to $\Q(\sqrt{-1})$.                      & \multirow{2}{*}{CM and Superrectangular} \\ \cline{1-2}
$\infty\circled{6}$                  & $V$ homothetic to $\Q(\sqrt{-3})$.                      &                    \\ \hline
\multirow{2}{*}{\circled{4}} & $V$ a non-CM Superrectangular space.                      & Superrectangular                  \\ \cline{2-3} 
                   & $V$ homothetic to one of the dodecagonal spaces.                      & \multirow{3}{*}{} \\ \cline{1-2}
\circled{3}+\circled{2}                  & \multirow{2}{*}{Expected elliptic families and a finite list.}     &                    \\ \cline{1-1}
\circled{2}+\circled{2}+\circled{2}                  &                        &                    \\ \hline
\end{tabular}
\end{center}

\begin{rem}[An application to euclidean geometry]\label{rem:applicazione.euclidea}
 It is worth noting that, as a consequence of the classification of all spaces of type \circled{4} obtained in this section, it can be shown that the red parallelogram in Figure~\ref{fig:dodecagonali:tau1} is the only parallelogram with the properties that:
 \begin{enumerate}
  \item all angles determined by sides and diagonals are rational multiples of $\pi$,
  \item it is neither a rectangle nor a rhombus. 
 \end{enumerate}
 Quadrilaterals with property (1) have already been considered in the literature (see for example \cite{Rigby}), and are related to intersecting triples of diagonals in regular polygons, a topic fully analysed in \cite{PoonenRubinstein}.
 
 A related result that arises from the argument concerns rational products of tangents of rational angles, in a way similar but distinct from the study carried out in \cite{MyersonProductsSines}.
\end{rem}

\section{Example of spaces corresponding to rational points on a curve genus 1}\label{sec:esempio.genere.1}
As an example of the possible phenomena that can appear in the more general case \circled{2}+\circled{2}+\circled{2} we show here an infinite family of spaces parametrized by the rational points on an elliptic curve of rank one over $\Q$.

Setting $x=e^{i \pi/2}, y=e^{i \pi /4}, z=e^{i \pi 3/4}$ in equation \eqref{equation:caso222:B} we obtain 
\begin{multline}
\sqrt{2}(a c - b d) (a b - 2 b c + c d)=\\
-a^2 b^2+a^2 b c-a^2 b d-a^2 c d+a b^2 c+a b^2 d-3 a b c^2+\\
+6 a b c d-a b d^2+a c^2 d-a c d^2-3 b^2 c d+b c^2 d+b c d^2-c^2d^2.\\
\end{multline}
 Let us take one of the two factors multiplying $\sqrt{2}$ and set it to zero. Any rational solution of the system
 
 \begin{equation}\label{equation:fam.ellittica:sistema}
  \begin{cases}
   a d -2 b c + c d=0\\
   -a^2 b^2+a^2 b c-a^2 b d-a^2 c d+a b^2 c+a b^2 d-3 a b c^2+\\
   +6 a b c d-a b d^2+a c^2 d-a c d^2-3 b^2 c d+b c^2 d+b c d^2-c^2d^2=0
  \end{cases}
\end{equation}
with $a,b,c,d$ distinct and nonzero gives a solution to equation \eqref{equation:caso222:B}, and thus a space of type  \circled{2}+\circled{2}+\circled{2}.
The system \eqref{equation:fam.ellittica:sistema} defines a variety in $\P_3$, which is the union of the two lines $\{a=d=0\}$ and $\{b=c=0\}$ and an irreducible curve $\Ci$ of genus 1.
 
After eliminating $c$ and applying the transformation 
\[
 \begin{cases}
   a=b+u\\
   d=b+v
 \end{cases}
\]
this curve has equation
\begin{equation}\label{eq:fam.ellittica.forma.buv}
 2 b^2 u^2 + 2 b^2 v^2 + 4 b u v^2 - u^2 v^2 + v^4=0.
\end{equation}
It is possible to put this plane curve in Weierstrass form, sending at infinity the point $(b:u:v)=(0:1:1)$.
Under the transformation
\[
 \begin{cases}
  b=\dfrac{-4 - 2 X + 2 X^2 + X^3 + 4 Y + X Y - Y^2}{4 + 4 X + 2 X^2 - 4 Y - 
 2 X Y + Y^2}\\
 u=-\dfrac{-2 + Y}{2 + 2 X - Y}\\
 v=1
 \end{cases}
\]
we obtain the Weierstrass form
\[
 Y^2= X^3+4X^2+6X+ 4.
\]
This elliptic curve has $j$-invariant 128. It has a $2$-torsion point $(-2,0)$. The rank of the Mordell-Weil group is $1$, with generator $(-1,1)$.

This curve has infinitely many rational points, and each quadruple $(a,b,c,d)$ provides a value of $\tau$ such that the space $\generatedQ{1,\tau}$ is of type \circled{2}+\circled{2}+\circled{2}. Let us see now that the set of spaces so obtained is infinite also considering them up to equivalence.

The value of $\tau$, expressed in the original coordinates $a,b,c,d$, is given by
\[
 \tau=\frac{(1 - i) (a b - c d)}{\sqrt{2} (a - b - \sqrt{2} b - c + \sqrt{2} c + d)}.
\]
If $a,b,c,d\in\Q$ we have that $\tau\in\Q(i,\sqrt{2})$, and we can consider its Galois conjugates over $\Q$. Let us call these four values $\tau_1,\tau_2,\tau_3,\tau_4$ (corresponding, in the same order, to the identity and the Galois automorphisms fixing $\sqrt{2},i,i\sqrt{2}$).

Let us recall that the cross-ratio of four complex numbers $z_1,z_2,z_3,z_4$ is the rational function 
\[
 \rho(z_1,z_2,z_3,z_4)=\frac{(z_3-z_1)(z_4-z_2)}{(z_3-z_2)(z_4-z_1)}.
\]
For every Möbius transformation $\sigma\in\PGL_2(\C)$ we have
\[\rho(z_1,z_2,z_3,z_4)=\rho(\sigma(z_1),\sigma(z_2),\sigma(z_3),\sigma(z_4)).\]

If $\generatedQ{1,\tau}$ and $\generatedQ{1,\tau'}$ are homothetic spaces, we have that $\tau'=\sigma(\tau)$ for some $\sigma\in\PGL_2(\Q)$, and therefore $\rho(\tau_1,\tau_2,\tau_3,\tau_4)=\rho(\tau'_1,\tau'_2,\tau'_3,\tau'_4)$; similarly, if $\tau'=\overline{\tau}$, then $(\tau'_1,\tau'_2,\tau'_3,\tau'_4)=(\tau_2,\tau_1,\tau_4,\tau_3)$, and the cross-ratio is invariant under such a permutation of the variables.

Any $\sigma\in\PGL_2(\Q)$ commutes with Galois automorphisms, so we have that $\tau_2,\tau_3,\tau_4$ can be expressed as rational functions of $a,b,c,d$, and so can the cross-ratio $\rho(\tau_1,\tau_2,\tau_3,\tau_4)$. Computing its expression we obtain 
\[
 \phi(a,b,c,d)=\frac{2(a-b-c+d)^2}{a^2 - 2 a b + 3 b^2 - 2 a c - 2 b c + 3 c^2 + 2 a d - 2 b d -  2 c d + d^2}.
\]
This $\phi$ defines a rational function on the projective curve $\Ci$, and by what we argued above we see that, if $P,Q\in\Ci(\Q)$ are two rational points that give two equivalent spaces, then $\phi(P)=\phi(Q)$. Clearly the generic fibre of the function $\phi$ is finite, and this immediately proves our claim, that the rational points on $\Ci$ give rise to infinitely many pairwise non-equivalent spaces.

Notice that this happens because of the special shape of the function $\phi$; for a ``general'' rational function $f$, the equation $f(P)=f(Q)$ would define in $\Ci\times\Ci$ a reducible curve whose components other than the diagonal would normally have genus greater than one.  However for the function in question the equation $\phi(P_1)=\phi(P_2)$ cuts in $\Ci\times\Ci$ six irreducible curves of genus 1.

An even clearer picture of the geometry of this problem is obtained by putting $\Ci$ in the form \eqref{eq:fam.ellittica.forma.buv}. In the $(b:u:v)$ variables, the function $\phi$ is given by
\[
 \phi=\frac{2v^2}{2b^2+v^2}.
\]

\appendix
\section{The irreducibility of the surface defined by equation~\texorpdfstring{\eqref{equation:caso222:B}}{(3.15)}}
For fixed $x,y,z$, equation \eqref{equation:caso222:B} defined a surface $S\subseteq \P_3$, whose rational points correspond to lattices of type \circled{2}+\circled{2}+\circled{2} in which the three rational angles have fixed arguments; in this appendix we study more in depth its geometric properties.

A motivation for this study, other than its interest on its own, comes from the fact that we can reduce this problem of rational points on a surface to a problem of rational points on curves, arguing in a way similar to Section~\ref{section:equation:1angle}, with the important distinction that here the space is not fixed.

In fact, assume that $S$ is not defined over $\Q$. Then there is a Galois automorphism $\sigma$ such that $\sigma(S)\neq S$. If $S$ is irreducible, then all rational points on $S$ lie in $S\cap\sigma(S)$, which is a curve.

The aim of this appendix is to study the surface $S$, proving that $S$ is irreducible unless $x=y=z=-1$, and that $S$ is not defined over $\Q$ unless $x,y,z$ are either three fourth roots of unity or three sixth roots of unity.
\subsection{Elimination of two quadratic equations}\label{sec:eliminazione2}

In order to prove some geometrical properties of the variety defined by the equation \eqref{equation:caso222:B}, we study here in general the elimination of one variable from two quadratic equations.

Let 
\begin{equation*}%
\tau^2+A_1\tau+A_2=\tau^2+B_1\tau+B_2=0,
\end{equation*}
be two equations, where for the moment $\tau, A_1,A_2,B_1,B_2$ are elements of a field. In our application, these two equations will be of the form \ref{eq:nangle:due}.
On subtracting we obtain $(B_1-A_1)\tau=-(B_2-A_2)$, whence,  on  multiplying any of the equations by $(B_1-A_1)^2$ and substituting for $(B_1-A_1)\tau$, we obtain 
\begin{equation}\label{E.E}
\E:=(B_2-A_2)^2-A_1(B_1-A_1)(B_2-A_2)+A_2(B_1-A_1)^2=0.
\end{equation}
This is of course the resultant of the quadratic polynomials above. After some calculations we also  find
\begin{equation}\label{E.E2}
\E=B_2^2-A_1B_1B_2+(A_1^2-2A_2)B_2+A_2B_1^2-A_1A_2B_1+A_2^2.
\end{equation}

Let us now consider $A_1,A_2,B_1,B_2$ as independent variables over an algebraically closed field $\K$ of characteristic $0$,  giving the weight $i$ to $A_i,B_i$.  Note that the  expression \eqref{E.E2}  for $\E$ has decreasing weights in the variables $B_i$, whereas the whole expression is homogeneous of degree $4$ with respect to these weights. The total ordinary degree  in all the four variables is $3$ whereas the separate degrees in $A_1,A_2$ and $B_1,B_2$ are both equal to $2$.

\begin{prop}\label{P.irr} The polynomial $\E$ is irreducible over $\K$. More generally, it is irreducible as a polynomial in $B_2$ over any extension of $\K(A_1,A_2,B_1)$ not containing a square root of $A_1^2-4A_2$. 
\end{prop}

Of course there is a similar statement on replacing $B_i$ with $A_i$.

\begin{proof}  %

Let $R^2=A_1^2-4A_2$. Then from \eqref{E.E} we find that  the roots of $\E$ as a polynomial in $B_2$ are
\begin{equation*}%
\frac{A_1B_1-A_1^2+2A_2\pm R(B_1-A_1)}{2}.
\end{equation*}

This clearly yields what asserted. 
\end{proof}

Let now $x,y,z$ be fixed roots of unity, different from $1$. 

We view $A_1,A_2,B_1,B_2$ as variables on the affine $4$-space $\A^4$. Letting $a,b,c,d$ be new  variables, we define a regular map  $\pi:\A^4\to\A^4$ by
\begin{equation}\label{E.pi}
 \begin{aligned}
 A_1&=\left(\frac{y-x}{y-1}\right)a+\left(\frac{xy-1}{y-1}\right)b, & A_2&=x ab \\  
B_1&=\left(\frac{z-x}{z-1}\right)c+\left(\frac{xz-1}{z-1}\right)d, & B_2&=xcd.
\end{aligned}
 \end{equation}

This also corresponds to a ring homomorphism  $$\overline\Q[A_1,A_2,B_1,B_2]\subset \overline\Q[a,b,c,d].$$ 

In the sequel we shall view  $\E$ through this homomorphism as a polynomial $\E^*$ in $a,b,c,d$, with coefficients in $\Q(x,y,z)$; it is homogeneous of degree $4$.  
Similarly we think of $A_1,A_2,B_1,B_2$ as polynomials in $a,b,c,d$ as given by \eqref{E.pi}.
The equation $\E^*=0$ defines the surface $S$, while the equation $\E=0$ defines the surface $\pi(S)$.

{\bf Warning}: Here a  word of warning is needed since the  ring homomorphism is not  always injective.  It  is injective (i.e. the map $\pi$ is dominant, which in turn amounts to the algebraic independence of $A_1,A_2,B_1,B_2$ as given by \eqref{E.pi}) except when $x=-1$ and either $y$ or $z$ is $-1$.  In these cases we have $A_1B_1=0$ as a polynomial in $a,b,c,d$. We have $A_1=B_1=0$ precisely if $x=y=z=-1$.

For simplicity of notation we omit explicit reference to this fact  in what follows, which  should not create confusion. 

\medskip

Note that up to a factor in $\Q(x,y,z)$, $\E^*$ equals the polynomial $P$  defined by Table~\eqref{table:caso222:coeff}. %
More precisely $P=\frac{(y-1)^2(z-1)^2}{x}\E^*$.

\medskip

\begin{thm} For fixed $x,y,z$ the polynomial $\E^*$ is irreducible in $\overline\Q[a,b,c,d]$, unless $x=y=z=-1$. 
\end{thm}

We note that for $x=y=z=-1$ the polynomial $\E^*$ factors as $(ab-cd)^2$.

\begin{proof}  Set

 \begin{equation*}%
 \begin{aligned}
 a_1&:=\left(\frac{y-x}{y-1}\right)a, & b_1&:=\left(\frac{xy-1}{y-1}\right)b,  \\  
c_1&:=\left(\frac{z-x}{z-1}\right)c, &  d_1&:= \left(\frac{xz-1}{z-1}\right)d.
\end{aligned}
\end{equation*}

Suppose, to start with, that we are in 

{\bf Case 1}: If $x$ is different from all among $y^{\pm 1},z^{\pm 1}$, which amounts to $a_1b_1c_1d_1\neq 0$.  

Then we have
\begin{equation*}
A_1=a_1+b_1,\quad B_1=c_1+d_1,\qquad A_2=\lambda a_1b_1,\quad B_2=\mu c_1d_1.
\end{equation*}
where for this proof we set 
\begin{align*}
\lambda&=\frac{x(y-1)^2}{(y-x)(xy-1)}, & \mu&=\frac{x(z-1)^2}{(z-x)(xz-1)}.
\end{align*}

Under the present assumption  the above ring homomorphism is injective and  yields  actually a   field extension $\overline\Q(a_1,b_1,c_1,d_1)/\overline\Q(A_1,A_2,B_1,B_2)$ which is  is finite  Galois with group $G$ isomorphic to  the four-group $\Z/(2)\times \Z/(2)$, acting trivially on constants and capital variables and acting on the lowercase variables by transpositions on $a_1,b_1$ and $c_1,d_1$.  

{\bf Remark}: This action gives a certain easy action on the original variables $a,b,c,d$ but we do not need to make this explicit. This may be relevant when studying the rational points $(a:b:c:d)$, since the present action is not defined over $\Q$ (thinking of $a,b,c,d$ as defined over $\Q$). 

\medskip

Suppose that $F\in\overline\Q[a_1,b_1,c_1,d_1]$ is a   nontrivial  irreducible factor of $\E^*$; we may assume it is homogeneous of degree $1$ or $2$. 

For $g\in G$, we have that $F^g:=F\circ g^{-1}$ is also a factor of $\E^*$, since the latter is invariant by $G$.  Let $\sigma\in G$ act by fixing $a_1,b_1$ and switching $c_1,d_1$.  The fixed field of $\sigma$ is $\overline\Q(a_1,b_1,B_1,B_2)$.  Note that $(c_1-d_1)^2$ is fixed by $\sigma$ (and equals $B_1^2-4\mu^{-1}B_2$), whereas $\sigma(c_1-d_1)=-(c_1-d_1)$.  Hence $\overline\Q(a_1,b_1,c_1,d_1)=\overline\Q(a_1,b_1,B_1,B_2)(c_1-d_1)$. 

\medskip

Suppose first that $F^\sigma=kF$ for a constant $k$; this should be necessarily $\pm 1$ since $\sigma^2=1$. If $k=-1$, then $F$ vanishes on putting $c_1=d_1$, whereas $\E^*$ does not (the term $B_2^2$ dominates). Hence $k=1$. But then $F$ is invariant by $\sigma$, thus lies in $\C(a_1,b_1,B_1,B_2)$, and then necessarily in $\C(a_1,b_1)[B_1,B_2]$ (e.g. since $c_1,d_1$ are integral over $\C[B_1,B_2]$ which is integrally closed). Then $F$ divides $\E^*$ in $\C(a_1,b_1)[B_1,B_2]$ and  by the above Proposition we deduce that $A_1^2-4A_2$ is a square in $\C(a_1,b_1,B_1) $ and hence in $\C(a_1,b_1)$. But $A_1^2-4A_2=(a_1+b_1)^2-4\lambda a_1b_1$, hence this latter possibility leads to $\lambda=0,1$. On the other hand $\lambda=0$ is excluded by the present hypothesis, whereas $\lambda=1$ easily leads to $x=1$, again excluded.

Therefore, since $F,F^\sigma$ are both irreducible, non proportional and divisors of $\E^*$,   $FF^\sigma$ is a divisor of $\E^*$; since it lies in $\C(a_1,b_1)[B_1,B_2]$ (being a norm) and  since $\E$ is irreducible in this ring (because as above $A_1^2-4A_2$ is not a square in $\C(a_1,b_1,B_1)$), $FF^\sigma$ must be a constant multiple of $\E$. Setting $c_1=d_1=u$ we have that $\E(a_1,b_1,u,u)$ is a square in $\C[a_1,b_1,u]$ (because it is a norm from a quadratic extension generated by $(c_1-d_1)$, or else because $F,F^\sigma$ become equal after such substitution).   
Since it lies in $\C[A_1,A_2,u]$  it is either a square in this ring, or vanishes on setting $a_1=b_1$; this last fact is excluded by direct computation. 
Indeed, since $B_1(u,u)=2u$, $B_2(u,u)=\mu u^2$, we have
\begin{equation*}
\E(a_1,b_1,u,u)=\mu^2u^4-2\mu A_1u^3+\mu (A_1^2-2A_2)u^2+4A_2u^2-2A_1A_2u+A_2^2.
\end{equation*}
On the other hand, from \eqref{E.E} we see that the discriminant with respect to $A_2$ is, up to a nonzero square, $B_1^2(u,u)-4B_2(u,u)=4u^2(1-\mu)$.
As above we have $\mu\neq 1$ hence this discriminant cannot vanish, hence the polynomial cannot be a square.

\medskip

This concludes the discussion of Case 1, and therefore from now on we may  deal with

{\bf Case 2}: If $a_1b_1c_1d_1=0$.

By symmetry we may assume say that $x=y$, i.e. $a_1=0$ and $A_1=(x+1)b$. 

Let us also assume first  that $x\neq z^{\pm 1}$. 

\medskip

If $A_1=0$ (i.e. $x=-1$) then equation \eqref{E.E} proves what we need unless $B_1=0$, since $ab$ is not a square.  If also $B_1=0$ we have $z=-1$ whence  $x=y=z=-1$.

Hence let us assume $A_1\neq 0$. 

By Proposition \ref{P.irr} applied with $A_i$ in place of $B_i$, if $\E^*$ is reducible over $\C(a,b,c,d)$ as a polynomial in $A_2$, then (by Gauss lemma) either is divisible by $b$ or  $B_1^2-4B_2$ is a square in $\C(b,c,d)$.  This last possibility is excluded as above since we are presently assuming that $x\neq z^{\pm 1}$ (so $c_1d_1$ are still defined).

On the other hand the expression \eqref{E.E2} shows that $\E^*$ is not a multiple of $b$.

\medskip

We are left with the case $x=y$ and $x=z$ or $x=z^{-1}$, which are symmetric, so assume $x=y=z\neq -1$. Then we have $B_1=(x+1)d$, $B_2=xcd$, and we directly can check the irreducibility of $\E^*$, e.g. checking it is irreducible in $c$, since it is quadratic in this variable and with  discriminant a constant multiple of $c^2(B_1-A_1)^2(A_1^2-4A_2)=(x+1)^2c^2(d-b)^2((x+1)^2b^2-4xab)$, not a square. 
\end{proof}

We remark that the irreducibility of the surface $S$ could also be used to argue for the finiteness of the rational angles in a fixed non-CM lattice, but the proof of Theorem~\ref{thm:finiti2} gives an explicit bound.

\begin{rem}
 We remark that the varieties $S$ and $\pi(S)$ are rational. Indeed one can set $Z_1=B_1-A_1$ and $Z_2=B_2-A_2$, and the equation \eqref{E.E} becomes $Z_2^2-A_1Z_1Z_2+A_2Z_1^2$, which is linear in $A_1,A_2$; this shows that $\pi(S)$ is rational. As for the variety $S$, one can see that $R^2=A_1^2-4A_2$, through the substitutions \eqref{E.pi}, gives rise to a ternary quadratic form $Q(a,b,R)=0$. After dehomogenizing with respect to, say, $a$ we can then parametrize rationally $b,R$ in terms of a parameter $\alpha$. Through \eqref{E.E} we can also express $R$ in terms of $c,d$, which leads to an equation which is linear in both $c,d$, thus allowing, for example, to express $d$ as a rational function of $\alpha$ and $c$.
 
 In any case the fact that these varieties are rational is not relevant to our arguments.
\end{rem}

\subsection{The field of definition of \texorpdfstring{$S$}{S}}\label{section:appendix:def.over.Q}
We observed that, when $S$ is not defined over $\Q$, its rational points lie in the intersection of the conjugates of $S$, and therefore on a union of finitely many varieties of dimension at most 1. With the following proposition we show that the abundance of rational angles on CM spaces is instead explained by the fact that they correspond to angles for which the surface $S$ is defined over $\Q$.
\begin{prop}
 The surface $S$ is defined over $\Q$ if and only if either $x^4=y^4=z^4=1$ or $x^6=y^6=z^6=1$.
\end{prop}
\begin{proof}
In order to understand when the surface $S$ is defined over $\Q$
one can consider the following four monomials in the variables $a,b,c,d$ from equation~\eqref{equation:caso222:B}, with their coefficients
\begin{align*}
 &a^2bc & -&(y-1)(z-1)(x-y)(x-z)\\
 &a^2bd & &(y-1)(z-1)(x-y)(xz-1)\\
 &ab^2c & &(y-1)(z-1)(x - z) (x y-1)\\
 &ab^2d & -&(y-1)(z-1)(xy -1)(xz-1).
\end{align*}
If the variety is defined over $\Q$, then the ratio of any two of those four coefficients, when they are not zero, must be a rational number.
Considering the first two, we see that either there is a rational relation between $1,x,z,xz$, or $x=y$; considering also the second two we obtain that either there is a rational relation between $1,x,z,xz$, or $x=y=-1$.
Similarly, pairing the first with the third and the second with the fourth, we obtain  that either there is a rational relation between $1,x,y,xy$, or $x=z=-1$.
If $x=y=-1$ the equation defining the surface becomes 
\[-4 a^2 b^2 (z-1)^2 - 4 c^2 d^2 (z-1)^2 - 32 a b c d z + 
 4 a b c^2 (1 + z)^2 + 4 a b d^2 (1 + z)^2,\]
 and for it to be defined over $\Q$, $z$ must have order 4 or 6 (and analogously if $x=z=-1$).
 If both a $\Q$-linear relation with non-zero coefficients between $1,x,z,xz$ and between $1,x,y,xy$ exist, then by Theorem~\ref{thm:unit.equation} $x,y,z$ are roots of unity of common order $4$ or $6$.
\end{proof}

\bibliography{DvoVenZan.bib}

\end{document}